\documentclass[preprint,12pt]{elsarticle}
\usepackage{amsmath,amssymb,amsthm,amsfonts,amscd,mathtools} 
\usepackage{graphics}                 
\usepackage{color}                   
\usepackage{hyperref}                
\usepackage{enumerate}
\usepackage{url}         
\usepackage{colonequals} 
\usepackage{algorithmic}
\usepackage{algorithm}

\newtheorem{theorem}{Theorem}[section]
\newtheorem{proposition}[theorem]{Proposition}

\newtheorem{lemma}[theorem]{Lemma}
\theoremstyle{definition}
\newtheorem{remark}[theorem]{Remark}

\newtheorem{definition}[theorem]{Definition}

\def\!{\mathop{\mathrm{!}}}

\def\R{\mathbb{ R}}
\def\H{\mathbb{H}}

\def\N{\mathcal{ N}}
\def\G{\mathcal{ G}}
\def\O{\mathcal{ O}}
\def\P{\mathcal{P}}
\def\tr{\mathrm{tr}}

\def\L{\mathcal{L}}
\def\norm#1{{\left\|\,#1\,\right\|}}
\def\inprod#1#2{\langle #1,\,#2\rangle}
\def\mc{\multicolumn}
\definecolor{blue}{rgb}{0,0,0.9}

\definecolor{red}{rgb}{0.9,0,0}

\journal{}

\begin{document}

\begin{frontmatter}
\title{Regularization of Wasserstein barycenters for $\varphi$-exponential distributions}
\author{S. Kum}
\address[S. Kum]{Department of Mathematics Education, Chungbuk National University, Cheongju 28644, Republic of Korea. Email: shkum@chungbuk.ac.kr.}
\author{M. H. Duong}
\address[M. H. Duong]{School of Mathematics, University of Birmingham, B15 2TT, UK. Email: h.duong@bham.ac.uk (corresponding author).}
\author{Y. Lim}
\address[Y. Lim]{Department of Mathematics, Sungkyunkwan University, Suwon 440-746, Republic of Korea. Email: ylim@skku.edu.}
\author{S. Yun}
\address[S. Yun]{Department of Mathematics Education, Sungkyunkwan University,  Seoul 03063, Republic of Korea. Email: yswmathedu@skku.edu.}
\begin{abstract}
In this paper, we focus on the analysis of the regularized Wasserstein barycenter problem. We provide uniqueness and a characterization of the barycenter for two important classes of probability measures: (i) Gaussian distributions and (ii) $q$-Gaussian distributions; each regularized
by a particular entropy functional. We propose an algorithm based on gradient projection method in the space of matrices in order to compute these regularized barycenters. We also consider  a general class of $\varphi$-exponential measures, for which only the
non-regularized barycenter is studied. Finally, we numerically show the influence of parameters and stability of the algorithm under small perturbation of data.
\end{abstract}
\end{frontmatter}
\section{Introduction}
\subsection{Regularization of barycenters in the Wasserstein space} In this paper we are interested in \textit{the regularization of barycenters in the Wasserstein space}, which is a minimization problem of the form
\begin{equation}
\label{eq: penalization}
\min_{\mu \in \P_2(\R^d)}\sum_{i=1}^n\frac{1}{2}\lambda_i W_2^2(\mu,\mu_i)+\gamma F(\mu),
\end{equation}
where $\P_2(\R^d)$ is the Wasserstein
space of probability measures on $\R^d$ with finite second moments;
$\{\mu_i\}_{i=1}^{n}$ are $n$ given probability measures in $\P_2(\R^d)$;
$W_2$ is the $L^2$-Wasserstein distance between two probability
measures in $\P_2(\R^d)$ (cf. Section \ref{sec: preliminary}),
and $F:\P_2(\R^d)\rightarrow \R$ is an entropy
functional. Finally $\gamma\geq 0$ is a given regularization parameter; $\lambda_1,\ldots, \lambda_n$
are given non-negative numbers (weights) satisfying
$\sum_{i=1}^n\lambda_i=1$.

\subsection{Literature review} Problem \eqref{eq: penalization}
 for $\gamma=0$ has been studied intensively in the literature.
 It was first studied by Knott and Smith \cite{KnottSmith1984} for Gaussian measures.
 In \cite{Agueh2011}, Agueh and Carlier studied the general case proving,
 among other things, the existence and uniqueness of a minimizer provided that one of $\mu_i$'s
  vanishes on small sets (i.e. sets whose Hausdorff dimension is at most $d-1$). Examples of such measures include those that are absolutely continuous with respect to the Lebesgue measure. The minimizer
  is called the barycenter of the measures $\mu_i$ with weights $\lambda_i$ extending a classical
  characterization of the Euclidean barycenter. The article \cite{Agueh2011} has sparked off many
   research activities from both theoretical and computational aspects over the last years.
    Wasserstein barycenters in different settings, such as over compact Riemannian manifolds \cite{KimPass2017} and over discrete data \cite{AnderesBorgwardtMiller2016}
     have been investigated . In the compact Riemannian setting, the condition to vanish on small sets ensuring uniqueness is replaced
by absolute continuity with respect to the volume measure  \cite{KimPass2017}. However, in the discrete setting, the uniqueness and absolute continuity of the barycenter is lost \cite{AnderesBorgwardtMiller2016}. Connections between Wasserstein barycenters
     and optimal transports have been explored \cite{Pass2012, KimPass2015}. Several computational methods for the computation
      of the barycenter have been developed \cite{CutiriDoucet2014, Alvarez2016, KumYun2018,PeyreCuturi2019}. Recently Wasserstein
      barycenters has found many applications in statistics, image processing and machine learning
       \cite{Rabin2011, MallastoFeragen2017, Srivastava2018}. We refer the reader to the mentioned papers and references therein
       for a more detailed account of the topic.

The case $\gamma>0$ has been studied in the recent paper
 \cite{BigotCazellesPapadakis2019} where the existence,
 uniqueness and stability of a minimizer, which is called the regularized barycenter,
  has been established. In particular, this paper shows that if the regularizing function is a proper and lower semicontinuous function (for the Wasserstein distance) and is strictly convex on its domain, then there exists a unique regularized barycenter even in the case of discrete measures. In addition,  the regularization parameter $\gamma$ was proved to provide smooth
barycenters especially when the input probability measures are
irregular which is useful for data analysis
\cite{BCP2018,Redko2019}. In addition, the regularized barycenter
problem also resembles the discretization formulation of Wasserstein
gradient flows for dissipative evolution equations
\cite{JKO1998,AGS08,Carlier2017} and the fractional heat equation
\cite{DuongJin2019} at a given time step where $\{\mu_i\}$ represent
discretized solutions at the previous steps and $\gamma$ is
proportional to the time-step parameter. 

Gaussian measures play an important role in the study of Wasserstein barycenter problem since in this case an useful characterization of the barycenter exists \cite{Agueh2011,Bhatia2018} which gives rise to efficient computational algorithms such as the fixed point approach \cite{Alvarez2016} and the gradient projection method \cite{KumYun2018}. Our aim in this paper is to seek for a large class of probability measures so that the regularized barycenter can be explicitly characterized and computed similarly to the case of Gaussian measures. It is worth mentioning that many papers in the literature study a related problem of  entropic regularization of optimal transports where the Wasserstein distance is regularized by an entropic term. The problem of finding a closed form solution for such problems in the case of Gaussian distributions has increasingly attracted interest in the community of computational optimal transport and machine learning \cite{Janati2020,Mallasto2021}. The problem that we study in this paper is different from these papers since the entropy term is added outside of the Wasserstein distance.

We will study the regularization problem
\eqref{eq: penalization} for two important classes of probability measures, namely Gaussian and $q$-Gaussian measures, where the entropy functional is the negative Boltzmann entropy and the Tsallis entropy, respectively. In addition, we also study the non-regularization problem (i.e., \eqref{eq: penalization} with $\gamma=0$) for the class of $\varphi$-exponential measures, which contains both Gaussian measures and  $q$-Gaussian measures as special cases, cf. Section \ref{sec: intro varphi-exponential} below. To state our main results, we now briefly recall the definition of $\varphi$-exponential measures; more detailed will be given in Section \ref{sec: preliminary}.

\subsection{$\varphi$-exponential distributions} 
\label{sec: intro varphi-exponential}
Let $\varphi$
 be an increasing, positive, continuous function on $(0,\infty)$, the $\varphi$-logarithmic is defined by \cite{Naudts}
\begin{equation}
\label{eq: varphi log}
\ln_{\varphi}(t):=\int_1^t\frac{1}{\varphi(s)}\,ds,
\end{equation}
which is increasing, concave and $C^1$ on $(0,\infty)$. Let $l_\varphi$ and $L_\varphi$ be respectively the infimum and the supremum of $\ln_\varphi$, that is
\begin{align*}
&l_\varphi:=\inf\limits_{t>0}\ln_\varphi(t)=\lim\limits_{t\downarrow 0} \ln_{\varphi}(t)\in [-\infty,0), 
\\& L_\varphi:=\sup\limits_{t>0}\ln_\varphi(t)=\lim\limits_{t\uparrow \infty} \ln_{\varphi}(t)\in (0,+\infty).
\end{align*}
The function $\ln_\varphi$ has the inverse function, which is called the $\varphi$-exponential function, and is defined on $(l_\varphi,L_\varphi)$. This inverse function can be extended to the whole $\R$ as
\begin{equation}
\label{eq: varphi exp}
\exp_{\varphi}(s):=\begin{cases}
0\qquad\text{for}~~s\leq l_\varphi,\\
\ln_\varphi^{-1}(s)\qquad\text{for}~~ s\in(l_\varphi,L_\varphi),\\
\infty\qquad\text{for}~~ s\geq L_\varphi,
\end{cases}
\end{equation}
which is $C^1$ on $(l_\varphi,L_\varphi)$.

Let ${\Bbb S}(d,\R)_+$ be the set of symmetric positive definite matrices of order $d$. Let $v\in
\mathbb{R}^d$ be a given vector and $V\in {\Bbb S}(d,\R)_+$ be a given
symmetric positive definite matrix. The $\varphi$-exponential
measure with mean $v$ and covariance matrix $V$, denoted by $G_\varphi(v,V)$, is the probability
measure on $\mathbb{R}^d$ with Lebesgue density
\begin{equation}
\label{eq: gvarphi}
g_\varphi(v,V)(x):=\exp_{\varphi}(\lambda_\varphi-c_\varphi|x-v|_V^2)\Big(\det(V)\Big)^{-\frac{1}{2}},
\end{equation}
where $|x|_V^2:=\langle x, V^{-1}x\rangle$, $\lambda_\varphi$ and $c_\varphi$ are normalization constants. Two important examples of $\varphi$-exponential
measures include Gaussian measures and $q$-Gaussian measures
corresponding to $\varphi(s)=s$ and $\varphi(s)=s^q$ respectively.
The $\varphi$-exponential measures play an important role in
statistical physics, information geometry and in the analysis of
nonlinear diffusion equations \cite{Otto2001,Ohara2009,Takatsu2012,
Takatsu2013}. More information about $\varphi$-exponential measures
will be reviewed in Section \ref{sec: preliminary}.
\subsection{Main results of the paper}
As already mentioned, in this paper we study the regularization problem \eqref{eq: penalization} for Gaussian measures and $q$-Gaussian measures, where the entropy functional is the (negative) Boltzmann entropy functional and the Tsallis entropy functional respectively, as well as the non-regularization problem for $\varphi$-exponential distributions. Main results of the present paper are explicit characterizations of the minimizer of \eqref{eq: penalization} and properties of the objective functions that can be summarized as follows.
\begin{theorem}\
\label{thm: main theorem 1}
\begin{enumerate}
\item Suppose that for each $i=1,\ldots, n$, $\mu_i$ is a $q$-Gaussian measure (Gaussian measure when $q=1$) with mean zero and covariance matrix $A_i\in \mathbb{S}(d,\R)_+$. Then the regularized barycenter problem \eqref{eq: penalization} has a unique minimizer, which is also a $q$-Gaussian measure with mean zero and covariance matrix $X$ satisfying
$$
X-\gamma m(q,d)(\det X)^{\frac{q-1}{2}} I=\sum_{i=1}^n \lambda_i \Big(X^\frac{1}{2}A_i X^\frac{1}{2}\Big)^\frac{1}{2},
$$
where $m(q,d)$ is a constant depending on $q$ and $d$ (see Theorem \ref{thm: q-Gaussian} for its explicit formula, in particular $m=1$ when $q=1$).
\item Suppose that for each $i=1,\ldots, n$, $\{\mu_i\}$ is a $\varphi$-exponential measure with mean zero and covariance matrix $A_i$. Then the unregularized barycenter problem (i.e. $\gamma=0$ in \eqref{eq: penalization}) has a unique minimizer, which is also a $\varphi$-exponential measure with mean zero and covariance matrix $X$ satisfying
$$
X=\sum_{i=1}^n \lambda_i (X^\frac{1}{2} A_i X^{1/2})^\frac{1}{2}.
$$
\end{enumerate}
\end{theorem}
\begin{theorem}
\label{thm: main theorem 2}
Suppose that $\{\mu_i\}$ are all Gaussian measures or all $q$-Gaussian measures with mean zero. Then the gradient of the objective function in the minimization problem \eqref{eq: penalization} is Lipschitz continuous, where the Lipschitz constant in each case can be found explicitly (see Theorem \ref{thm: Lipschitz continuity} and Theorem \ref{thm: Lipschitz continuity q-Gaussians} respectively).
\end{theorem}
Theorem \ref{thm: main theorem 1} summarizes Proposition \ref{prop: invariant}, Theorem \ref{thm: Gaussian} (for Gaussian measures), Theorem \ref{thm: q-Gaussian} (for $q$-Gaussian measures) and Theorem \ref{thm: varphi} (for general $\varphi$-exponential measures). Theorem \ref{thm: main theorem 2} summarizes Theorem \ref{thm: Lipschitz continuity} (for Gaussian measures)
and Theorem \ref{thm: Lipschitz continuity q-Gaussians} (for $q$-Gaussian measures). 

The key to the analysis of the present paper is that  the spaces of $\varphi$-exponential measures
and  Gaussian measures are isometric in the sense of Wasserstein
geometry \cite{Takatsu2012, Takatsu2013}, that is
\begin{equation*}
W_2(G_\varphi(v,V),G_\varphi(u, U))=W_2(\N(v, V),\N(u,U)),
\end{equation*}
where $\N(v, V)$ denotes a Gaussian measure with mean $v$ and covariance matrix $V$.
Therefore, since the Wassertein distance between Gaussian measures
can be computed explicitly, the objective functional in \eqref{eq:
penalization} can also be computed explicitly in terms of the
covariance matrices and \eqref{eq: penalization} becomes a minimization
problem over the space of symmetric positive definite matrices. We
then prove the strict convexity of the objective function and the
existence of solutions to the optimality equation using matrix
analysis tools as in \cite{Bhatia2018}. Theorems \ref{thm:
Gaussian},   \ref{thm: q-Gaussian} and \ref{thm: varphi} establish
the existence and uniqueness of a minimizer and provide an explicit
characterization of the minimizer in terms of nonlinear matrix
equations for the covariance matrix generalizing the characterization of the
Wasserstein barycenter for Gaussian measures in
\cite{Agueh2011,Bhatia2018} to the regularized Wasserstein barycenter
for Gaussian measures, $q$-Gaussian measures, and $\varphi$-exponential measures. Theorem
\ref{thm: Lipschitz continuity} and Theorem \ref{thm: Lipschitz
continuity q-Gaussians} prove the Lipschitz continuity of the
gradient of the objective function providing an explicit upper bound
for the Lipschitz constant generalizing the results of
\cite{KumYun2018} for the barycenter for Gaussian measures to our
setting. We also perform numerical experiments to show the affect of the parameter $q$ and a stability property of the algorithm under small perturbation of the data, cf. Section \ref{sec: numerical experiments}.

\subsection{Organization of the paper}
The rest of the paper is organized as follows. In Section \ref{sec: preliminary} we review relevant knowledge that will be used in subsequent sections on the Wasserstein metric and the Wasserstein geometry of Gaussian and $\varphi$-exponential distributions. Then we study the regularization of barycenters for Gaussian measures in Section \ref{sec: Gaussians} and extend these results to $q$-Gaussian and $\varphi$-exponential measures in Section \ref{sec: q-Gaussians} and Section \ref{sec: vaphi expo measures}. In Section \ref{sec: GPM} we describe a gradient projection method for the computation of the minimizer and prove that the gradient function is Lipschitz continuous. Finally, in Section \ref{sec: numerical experiments}, we numerically show affect of parameters to the minimizer and stability of the algorithm under small perturbation of data. 
\section{Wasserstein metric, Gaussian measures and $\varphi$-exponential measures}
\label{sec: preliminary}
In this section, we summarize relevant knowledge that will be used in subsequent sections on the Wasserstein metric and the Wasserstein geometry of Gaussian and $\varphi$-exponential distributions.
\subsection{Wasserstein metric}
We recall that $\P_2(\R^d)$ is the space of probability measures $\mu$ on $\R^d$ with finite second moment, namely
$$
\int_{\R^d}|x|^2\mu(dx)<\infty.
$$
Let $\mu$ and $\nu$ be two probability measures belonging to $\P_2(\R^d)$. The $L^2$-Wasserstein distance, $W_2(\mu,\nu)$, between $\mu$ and $\nu$ is defined via
\begin{equation}
\label{eq: W2}
W^2_2(\mu,\nu):=\inf_{\gamma\in \Gamma(\mu,\nu)}\int_{\R^d\times\R^d}|x-y|^2\,\gamma(dx,dy),
\end{equation}
where $\Gamma(\mu,\nu)$ denotes the set of transport plans between $\mu$ and $\nu$, i.e., the set of all probability measures on $\R^d\times \R^d$ having $\mu$ and $\nu$ as the first and the second marginals respectively. More precisely,
$$
\Gamma(\mu,\nu):=\{\gamma\in \P(\R^d\times \R^d): \gamma(A\times \R^d)=\mu(A)~\text{and}~ \gamma(\R^d\times A)=\nu(A)\},
$$
for all Borel measurable sets $A\subset\R^d$. It has been proved that, under rather general conditions (e.g., when $\mu$ and $\nu$ are absolutely continuous with respect to the Lesbegue measure), an optimal transport plan in \eqref{eq: W2} uniquely exists and is of the form $\gamma=[\mathrm{id}\times\nabla\psi]_\#\mu$ for some convex function $\psi$ where $\#$ denotes the push forward \cite{Brenier1991, GangboMcCann1996}.

The Wasserstein distance is an instance of a Monge-Kantorevich optimal transportation cost functional and plays a key role in many branches of mathematics such as optimal transportation, partial differential equations, geometric analysis and has been found many applications in other fields such as economics, statistical physics and recently in machine learning. We refer the reader to the celebrated monograph \cite{Villani2003} for a great exposition of the topic.

We now consider two important classes of probability measures, namely Gaussian measures and $\varphi$-exponential measures, for which there is an explicit expression for the Wasserstein distance between two members of the same class. Although Gaussian measures are special cases of $\varphi$-exponential measures, but we consider them separately since many proofs for the former are much simplified than those for the latter.
\subsection{Wasserstein distance of Gaussian measures}
Given any $X\in {\Bbb S}(d,\R)_+$, we define a symmetric positive definite matrix $X^{1/2}$ such that $X^{1/2} X^{1/2}=X$. Throughout the paper, we denote by $I$ the identity matrix of order $d$. The Wasserstein distance between two Gaussian measures is well-known \cite{GivensShortt1984}, see also e.g., \cite{Takatsu2012}:
\begin{equation}
\label{eq: W2-Gaussians}
W_2(\N(u,U),\N(v,V))^2=|u-v|^2+\mathrm{tr}U+\mathrm{tr}V-2\mathrm{tr}\sqrt{V^\frac{1}{2}U V^\frac{1}{2}}.
\end{equation}
Furthermore, $[\mathrm{id}\times \nabla \mathcal{T}]_\#\N(u,U)$ is the optimal plan between them, where
\begin{equation}
\label{eq: optimal plan Gaussian}
\mathcal{T}(x)=\frac{1}{2}\langle x-u, T(x-u)\rangle+\langle x,v\rangle, \quad T=V^\frac{1}{2}\Big(V^\frac{1}{2} U V^\frac{1}{2}\Big)^{-\frac{1}{2}}V^\frac{1}{2}.
\end{equation}
\subsection{The entropy of Gaussian measures}
The (negative) Boltzmann entropy of a probability measure $\mu=\mu(x)dx$ on $\mathbb{R}^d$ is defined by
\begin{equation}
\label{eq: entropy}
F(\mu):=\int_{\R^d} \mu(x) \log \mu(x)\,dx.
\end{equation}
Using Gaussian integral, the (negative) Boltzmann entropy of a Gaussian measure can be computed explicitly \cite[Theorem 9.4.1]{CoverThomas1991}:
\begin{equation}
\label{eq: entropy of Gaussian measures}
F(\N(u, U))=-\frac{d}{2}\ln(2\pi e)-\frac{1}{2}\ln\det(U).
\end{equation}
We now consider the second class of probability measures: $\varphi$-exponential measures.
\subsection{$\varphi$-exponential measures and Wassertein distance}
We recall that for a given increasing, positive and continuous function $\varphi$ on $(0,\infty)$, the $\varphi$-logarithmic function and the $\varphi$-exponential function are respectively defined in \eqref{eq: varphi log} and \eqref{eq: varphi exp}. Two important classes of $\varphi$-exponential functions are:
\begin{enumerate}[(i)]
\item  $\varphi(s)=s$: the $\varphi$-logarithmic function and the $\varphi$-exponential function become the traditional logarithmic and exponential functions: $\ln_\varphi(t)=\ln (t),~~\exp_{\varphi}(t)=\exp(t)$.
\item $\varphi(s)=s^q$ for some $q>0$: the $\varphi$-logarithmic function and the $\varphi$-exponential function become the $q$-logarithmic and $q$-exponential functions respectively
\begin{equation*}
\ln_{\varphi}(t)=\log_q(t)=\frac{t^{1-q}-1}{1-q} \quad \text{for}~ t>0,\quad \exp_{\varphi}(t)=\exp_q(t)=\Big( 1+(1-q)t\Big)_+^\frac{1}{1-q},
\end{equation*}
where $[x]_+=\max\{0,x\}$ and by convention $0^a\colonequals\infty$. The $q$-logarithmic function satisfies the following property
\begin{equation}
\label{eq: q-log properties}
\ln_q(xy)=\ln_q(x)+\ln_q(y)+(1-q)\ln_q(x)\ln_q(y).
\end{equation}
\end{enumerate}
\begin{definition} For any $a\in\R$, we define $\O(a)$ to be the set of all increasing,  positive, continuous function $\varphi$ on $(0,\infty)$ such that $\max\{\delta_\varphi,\delta^\varphi\}<a$ where
$$
\delta_\varphi:=\inf\Big\{\delta\in \R\Big\vert \lim\limits_{s\downarrow 0}\frac{s^{1+\delta}}{\varphi(s)}~\text{exists}\Big\}, \quad  \delta^\varphi:=\inf\Big\{\delta\in \R\Big\vert \lim\limits_{s\uparrow \infty}\frac{s^{1+\delta}}{\varphi(s)}=\infty\Big\}.
$$
\end{definition}
It is proved in \cite[Proposition
3.2]{Takatsu2013} that for any $\varphi\in \O(2/(d+2))$ there exist constants $\lambda_\varphi$ and $c_\varphi$ such that (cf. \eqref{eq: gvarphi} in the Introduction)
%
\begin{equation*}
g_\varphi(v,V)(x):=\exp_{\varphi}(\lambda_\varphi-c_\varphi|x-v|_V^2)\,\Big(\det(V)\Big)^{-\frac{1}{2}},
\end{equation*}
where $|x|_V^2:=\langle x, V^{-1}x \rangle$, is a probability density on $\R^d$ with mean $v$ and
 covariance matrix $V$, which is called a $\varphi$-exponential
 distribution. Note that, in the above expression, $\lambda_\varphi$ and
$c_\varphi$ are enough to define only at the identity matrix $I_d$,
not on all ${\Bbb S}(d,\R)_+$. We define the space of all $\varphi$-exponential distribution measures by
\begin{equation}
\G_{\varphi}:=\Big\{G_\varphi(v,V):=g_{\varphi}(v,V)\L^d\big\vert (v,V)\in \R^d\times {\Bbb S}(d,\R)_+\Big\}.
\end{equation}
Above $\L^d$ is the Lesbesgue measure on $\R^d$. Two important cases:
\begin{enumerate}[(i)]
\item $\varphi=s$, $\G_\varphi$ reduces to the class of Gaussian measures with mean $v$ and covariance matrix $V$.
\item In the case $\varphi=s^q$, $\G_{\varphi}$ becomes the class of all $q$-Gaussian measures
$$\G_q=\Big\{G_q(v,V)\big\vert (v,V)\in
\R^d\times {\Bbb S}(d,\R)_+\Big\}$$ where
\begin{equation*}
G_q(v,V)=C_0(q,d)(\det V)^{-\frac{1}{2}}\exp_q\Big(-\frac{1}{2}C_1(q,d)\langle x-v,V^{-1}(x-v)\rangle\Big)\L^d,
\end{equation*}
and $C_0(q,d), C_1(q,d)$ are given by
\begin{align*}
& C_1(q,d)=\frac{2}{2+(d+2)(1-q)},\\
&C_0(q,d)=
\begin{cases} \frac{\Gamma\Big(\frac{2-q}{1-q}+\frac{d}{2}\Big)}{\Gamma\Big(\frac{2-q}{1-q}\Big)} \Big(\frac{(1-q)C_1(q,d)}{2\pi}\Big)^\frac{d}{2}& \text{if $0<q<1$,}
\\ \frac{\Gamma\Big(\frac{1}{q-1}\Big)}{\Gamma\Big(\frac{1}{q-1}-\frac{d}{2}\Big)}\Big(\frac{(q-1)C_1(q,d)}{2\pi}\Big)^\frac{d}{2}& \text{if $1<q<\frac{d+4}{d+2}$.}
\end{cases}
\end{align*}
Note that $C_1(1,d)=1$ and $C_0(q,d)\rightarrow (2\pi)^{-d/2}$ as $q\rightarrow 1$, which follows from Stirling’s formula. Thus Gaussian measures are special cases of $q$-Gaussian measures.
\end{enumerate}
The $\varphi$-exponential measures play an important role in statistical physics, information geometry and in the analysis of nonlinear diffusion equations \cite{Otto2001,Ohara2009,Takatsu2012, Takatsu2013}. We refer to \cite{Ohara2009,Takatsu2012, Duong2015} for further details on $q$-Gaussian measures, $\varphi$-exponential measures and and their properties.

The following result explains why $q$-Gaussian measures and $\varphi$-exponential measures are special. It will play a key role in the analysis of this paper.
\begin{proposition}
\label{prop: q-Gaussian properties}
The following statements hold \cite{Takatsu2012,Takatsu2013}
\begin{enumerate}
\item For any $q\in (0,1)\cup \Big(1,\frac{d+4}{d+2}\Big)$, the space of $q$-Gaussian measures is convex and isometric to the space of Gaussian measures with respect to the Wasserstein metric.
\item For any $\varphi\in\O(2/(d+2))$ with $d\geq 2$, the space $\G_\varphi$ is convex and isometric to the space of Gaussian measures with respect to the Wasserstein metric.
\item Let $G_\varphi(\nu, V)$ and $G_\varphi(\mu, U)$ be two $\varphi$-exponential distributions. Then~ $[\mathrm{id}\times\nabla \mathcal{T}]_{\#}G_\varphi(\mu,U)$, where $\mathcal{T}$ is defined in \eqref{eq: optimal plan Gaussian}, is the optimal plan in the definition of $W^2_\varphi(G_q(\nu, V), G_\varphi(\mu, U)$.
\item We have
\begin{align} \label{eq: W2 formula}
W_2(G_\varphi(\mu,U),G_\varphi(\nu,V))^2&=W_2(G_q(\mu,U),G_q(\nu,V))^2\notag
\\&=W_2(\N(\mu,U),\N(\nu,V))^2\notag\\
&=|\mu-\nu|^2+\mathrm{tr}U+\mathrm{tr}V-2\mathrm{tr}\sqrt{V^\frac{1}{2}U
V^\frac{1}{2}}.
\end{align}
\end{enumerate}
\end{proposition}
\subsection{The Tsallis entropy of a $q$-Gaussian measure}
The Tsallis entropy of a probability measure $\mu=\mu(x)dx$ on $\mathbb{R}^d$ is defined by
\begin{equation}
\label{eq: Tsalli}
F_q(\mu):=\int_{\R^d} \mu(x)\ln_q\mu(x)\,dx=\frac{1}{1-q}\int_{\R^d}[\mu(x)^{1-q}-1]\mu(x)\,dx.
\end{equation}
The Tsallis entropy of a $q$-Gaussian can also be computed explicitly using the property \eqref{eq: q-log properties} and similar computations as in the Gaussian case.
\begin{lemma}
\label{lem: Tsallis entropy of q-Gaussian}
It holds that  \cite{Duong2015}
\begin{equation}
F_q(G_q(\mu, U))=-\frac{d}{2}C_1(q,d)+\Big[1-(1-q)\frac{d}{2}C_1(q,d)\Big]\ln_{q}\frac{C_0(q,d)}{(\det U)^\frac{1}{2}}.
\end{equation}
\end{lemma}
The first result of the present paper is the following proposition.
\begin{proposition}
\label{prop: invariant}
Suppose that $\mu_i\sim G_q(0, A_i)$. Then the regularized barycenter problem \eqref{eq: penalization} has a unique minimizer, which is also a $q$-Gaussian measure with mean $0$. This statement holds also for $q=1$ and in this case, the minimizer is a Gaussian measure with mean $0$. Similarly, when $\{\mu_i\}$ are all $\varphi$-exponential distributions with mean $0$, then the unregularized barycenter problem has a unique minimizer which is also a  $\varphi$-exponential distribution with mean $0$.
\end{proposition}
\begin{proof}
Since each of $\{\mu_i\}_{i=1}^n$ is a $q$-Gaussian measure with mean zero, then there exists a unique minimizer $\mu_*\in\mathcal{P}_2(\mathbb{R}^d)$, which is absolutely continuous with respect to the $d$-dimensional Lebesgue measure~ \cite{BigotCazellesPapadakis2019}. Let $v$ and $V$ be the mean and covariance matrix of $\mu_*$. Let $G_q(v,V)$ be the $q$-Gaussian measure with the same mean $v$ and covariance matrix $V$. 
Next we will show that
$$
\mu_*=G_q(v,V)\quad\text{and}\quad v=0\quad(\text{thus}~~\mu_*=G_q(0,V)).
$$
Since $G_q(v,V)$ minimizes the Tsallis entropy $F_q$ among all probability measures $\mu$ which are absolutely continuous with the $d$-dimensional Lebesgue measure having mean $v$ and covariance matrix $V$ (see for instance \cite{Takatsu2012}), we have
\begin{equation}
\label{eq: Tsallis inequal}
F_q(\mu_*)\geq F_q(G_q(v,V)).
\end{equation}
We recall the following equivalent, Monge and Kantorovich duality, characterizations of the Wassertein distance between two probability measures $\mu, \nu \in \P_2(\R^d)$ (see \cite[Theorem 5.10]{Villani2008})
\begin{align*}
W_2(\mu,\nu)^2&=\inf\limits_{T_\# \mu=\nu}\int_{\mathbb{R}^d}|x-T(x)|^2\,d\mu(x)
\\&=\sup_{\phi\in L^1(\nu)}\Big\{\int_{\mathbb{R}^d}\phi(y)^c\,d\nu(y)-\int_{\mathbb{R}^d}\phi(x)\,d\mu(x)\Big\},
\end{align*}
where $\phi^c(y)=\inf\limits_{y\in\mathbb{R}^d}\{\phi(x)+|x-y|^2\}$. In addition, the optimal transport map $T^*$ and the optimal Kantorovich potential $\phi^*$ in the above problems satisfy
$$
x-T^*(x)=\frac{1}{2}\nabla \phi^*(x).
$$
Let $T_i$ and $\phi_i$, $i=1,\ldots, n$ be the optimal transport map and the optimal Kantorovich potential for $W_2(\mu_i, G_q(v,V))$, that is
\begin{align*}
W_2(\mu_i,G_q(v,V))^2 &=\int_{\mathbb{R}^d}|x-T_i(x)|^2\,d \mu_i(x)
\\&=\int_{\mathbb{R}^d}\phi_i(y)^c \,d G_q(v,V)(y)-\int_{\mathbb{R}^d}\phi(x)\,d\mu_i(x).
\end{align*}
According to \cite[Theorem A]{Takatsu2012}, $T_i$ is given by $T_i=\nabla\mathcal{T}_i(x)$ where
\begin{align*}
\mathcal{T}_i(x)=\frac{1}{2}\langle x, \bar{T}_i x\rangle + \langle x, v\rangle, \quad  \bar{T}_i= V^{1/2}\Big(V^{1/2} A_i V^{1/2}\Big)^{-1/2} V^{1/2}.
\end{align*}
It follows that
$$
\phi_i(x)=|x|^2-2\mathcal{T}_i(x)=|x|^2-\langle x, \bar{T}_i x\rangle-2 \langle x, v\rangle.
$$
Therefore,
$$
\phi_i(y)^c=\phi_i(\bar{x})+\frac{1}{4}|\nabla \phi_i(\bar{x})|^2 \quad \text{where}\quad \nabla\phi_i(\bar{x})+2(\bar{x}-y)=0.
$$
It follows that the Jacobian matrix $J_i$ when changing the variable from $y$ to $\bar{x}$ is constant, $J_i=2I-\bar{T}_i$.
We have
\begin{align*}
\int_{\R^d}\phi_i^c(y) d\mu_*(y)&=\int_{\R^d}\Big(\phi_i(\bar{x})+\frac{1}{4}|\nabla \phi_i(\bar{x})|^2\Big)\,d\mu_*(y)
\\&\overset{(*)}{=}|J_i|\int_{\R^d}\Big(\phi_i(\bar{x})+\frac{1}{4}|\nabla \phi_i(\bar{x})|^2\Big)\,d\mu_*(\bar{x})
\\&\overset{(**)}{=}|J_i|\int_{\R^d}\Big(\phi_i(\bar{x})+\frac{1}{4}|\nabla \phi_i(\bar{x})|^2\Big)\,dG_q(v,V)(\bar{x})
\\&=\int_{\R^d}\phi_i^c(y) d G_q(v,V)(y),
\end{align*}
where $(**)$ follows from $(*)$ since the $(*)$ depends only on the mean and covariance of $\mu_*$ which is the same as $G_q(v,V)$.
Therefore
\begin{align}
W_2(\mu_i,\mu_*)^2&\geq \int_{\mathbb{R}^d}\phi_i(y)^c \,d \mu_*(y)-\int_{\mathbb{R}^d}\phi(x)\,d\mu_i(x)\notag
\\&=\int_{\mathbb{R}^d}\phi_i(y)^c d G_q(v,V)(y)-\int_{\mathbb{R}^d}\phi(x)\,d\mu_i(x)\notag
\\&=W_2(\mu_i,G_q(v,V))^2.\label{eq: W2inequal}
\end{align}
From \eqref{eq: Tsallis inequal} and \eqref{eq: W2inequal} we get
$$
\sum_{i=1}^n \lambda_i W_2(\mu_i,\mu_*)^2 +F_q(\mu_*)\geq \sum_{i=1}^n \lambda_i W_2(\mu_i,G_q(v,V))^2 +F_q(G_q(v,V)).
$$
By the uniqueness of minimizers, we deduce that $\mu_*=G_q(v,V)$. Moreover, the facts that
$$
F_q(G_q(v,V))=F_q(G_q(0,V)),\quad W_2(\mu_i, G_q(0,V))\leq W_2(\mu_i,G_q(v,V))
$$ 
ensure $v=0$. Note that this proof also holds true for $q=1$ where $q$-Gaussian measures and the Tsallis entropy are respectively replaced by Gaussian measures and the Boltzmann entropy. Similarly, using the third part of Proposition \ref{prop: q-Gaussian properties}, we can show that the minimizer of the unregularized barycenter is again a $\varphi$-exponential distribution if all the $\mu_i$ are $\varphi$-exponential distributions.  This completes the proof of this proposition.
\end{proof}
\section{Regularization of barycenters for Gaussian measures}
\label{sec: Gaussians}
In this section we study the following regularization of barycenters in the space of Gaussian measures
\begin{equation}
\label{eq: penalization Gaussian}
\min_{\mu\in \P_2(\R^d)}\sum_{i=1}^n\frac{1}{2}\lambda_i W_2^2(\mu,\mu_i)+\gamma F(\mu),
\end{equation}
where $\mu_i\sim \N(0,A_i)$ ($i=1,\ldots, n$),  $F$ is the (negative) Boltzmann entropy functional of a probability measure defined in \eqref{eq: entropy} and $\gamma>0$ is a regularization parameter. 

According to Proposition \ref{prop: invariant}, we only need to seek for the minimizer $\mu$ among Gaussian measures with mean zero, that is $\mu\sim \N(0,X)$ for some covariance matrix $X$. We note that we consider here Gaussian measures with zero mean just for simplicity, see Remark \ref{rem: non-zero gaussians} for further discussion on this assumption. The main results of the paper can be easily extended to the case of non-zero mean. From now on, we equip ${\Bbb S}(d,\R)_+$ with the
Frobenius inner product $\langle X,Y\rangle :=\tr(X^TY)$. The
Frobenius norm is defined by
$\|X\|_F=\Big(\tr(X^T X)\Big)^\frac{1}{2}$. For
$X,Y\in {\Bbb S}(d,\R),$ we write $X\leq Y$ if $Y-X$ is positive
semidefinite, and $X<Y$ if $Y-X$ is positive definite. Note that
$X\leq Y$ if and only if $\langle x,Xx\rangle\leq \langle
x,Yx\rangle$ for all $x\in {\Bbb R}^d.$ We denote $[X,Y]$ by the
L\"owner order interval $[X,Y]:=\{Z:X\leq Z\leq Y\}.$
\begin{theorem}
\label{thm: Gaussian}
Assume that $\{\mu_i\}$ are Gaussian distributions with mean zero and covariance matrix $A_i$, $\mu_i\sim \N(0,A_i)$ for $i=1,\ldots, n$. The regularization of barycenters problem \eqref{eq: penalization} has a unique solution $\mu\sim \N(0,X)$ where the covariance matrix $X$ solves the following nonlinear matrix equation
\begin{equation}
\label{eq: optimal equation Gaussian}
X-\gamma I=\sum_{i=1}^n \lambda_i (X^\frac{1}{2} A_i X^{1/2})^\frac{1}{2}.
\end{equation}
In particular, in the scalar case ($d=1$), we obtain
\begin{equation}
\label{eq: 1D}
X=\frac{\left[\sum_{i=1}^n\lambda_i A_i^\frac{1}{2}+\Big(\big(\sum_{i=1}^n\lambda_i A_i^\frac{1}{2}\big)^2+4\gamma\Big)^\frac{1}{2}\right]^2}{4}.
\end{equation}
\end{theorem}
Before proving this theorem, we show the existence of solutions to equation \eqref{eq: optimal equation Gaussian}.
\begin{lemma}
\label{lem: existence Gaussian}
Equation \eqref{eq: optimal equation Gaussian}
 has a positive definite solution.
\end{lemma}
\begin{proof}
 Pick
$0<\alpha_0<\beta_0$ so that $\alpha_0 I\leq A_i\leq \beta_0 I$ for
all $i=1,\dots, n.$ Set
$$
\alpha_*:=\left(\frac{\sqrt{\alpha_0}+\sqrt{\alpha_0+4\gamma}}{2}\right)^2,
\qquad \beta_*:=\left(\frac{\sqrt{\beta_0}+\sqrt{\beta_0+4\gamma}}{2}\right)^2.
$$
Then for matrices $X$ satisfying $\alpha_* I\leq X\leq \beta_* I$ we have,
$$\alpha_0 X\leq X^{1/2}A_iX^{1/2}\leq \beta_0 I, \qquad i=1,\dots,
n$$ and hence
$$\sqrt{\alpha_0}\sqrt{\alpha_*} I\leq \sqrt{\alpha_0}X^{1/2}\leq
(X^{1/2}A_iX^{1/2})^{1/2}\leq \sqrt{\beta_0}X^{1/2}\leq
\sqrt{\beta_0}\sqrt{\beta_*}I.$$ By definition of $\alpha_*$ and
$\beta_*$,
\begin{eqnarray*}
\alpha_* I&=&\sqrt{\alpha_0}\sqrt{\alpha_*}I+\gamma
I\leq\sum_{i=1}^n\lambda_i(X^{1/2}A_iX^{1/2})^{1/2}+\gamma I\\
&\leq&\sqrt{\beta_0}\sqrt{\beta_*}I+\gamma I=\beta_* I
\end{eqnarray*}
for every $X\in [\alpha_* I,\beta_* I]:=\{Z: \alpha_* I\leq Z\leq
\beta_* I\}.$ This shows that the map
$$
f(X):=\sum_{i=1}^n\lambda_i(X^{1/2}A_iX^{1/2})^{1/2}+\gamma I
$$ 
is a continuous self map on the L\"owner order interval $[\alpha_*
I,\beta_* I].$ By Brouwer's fixed point theorem, it has a fixed
point.
\end{proof}
We are now ready to prove Theorem \ref{thm: Gaussian}
\begin{proof}[Proof of Theorem \ref{thm: Gaussian}]
According to \eqref{eq: W2-Gaussians} and \eqref{eq: entropy of Gaussian measures} we have
\begin{align*}
W_2^2(\mu_i, \mu)&=\tr X+\tr A_i-2\tr\Big(A_i^\frac{1}{2}XA_i^\frac{1}{2}\Big)^\frac{1}{2},
\\ F(\mu)&=-\frac{d}{2}\ln(2\pi e)-\frac{1}{2}\ln(\det X).
\end{align*}
Thus we can write \eqref{eq: penalization} as a minimization problem in the space of symmetric positive definite matrices
\begin{equation}
\label{eq: penlization2}
\min_{X\in \mathbb{S}(d,\mathbb{R})_+} \frac{1}{2} f(X)
\end{equation}
where
\begin{align}
\label{eq: f}
f(X)&:=\sum_{i=1}^n \lambda_i\tr A_i+\sum_{i=1}^n \lambda_i \tr\Big(X-2\big(A_i^\frac{1}{2}X A_i^\frac{1}{2}\big)^\frac{1}{2}\Big)
-\gamma\ln \det(X)-\gamma d \ln(2\pi e)\notag
\\&:=f_1(X)+\gamma f_2(X),
\end{align}
where
\begin{align*}
&f_1(X)=\sum_{i=1}^n \lambda_i\tr A_i+\sum_{i=1}^n \lambda_i \tr\Big(X-2\big(A_i^\frac{1}{2}X A_i^\frac{1}{2}\big)^\frac{1}{2}\Big),
\\& f_2(X)=-\ln \det(X)-d \ln(2\pi e).
\end{align*}
It has been proved \cite{Bhatia2018} that
\begin{enumerate}[(i)]
\item $X\mapsto f_1(X)$ is strictly convex,
\item $Df_1(X)(Y)=\tr\Big(I-\sum_{i=1}^n \lambda_i (A_i\sharp X^{-1})\Big) Y$,
\end{enumerate}
where $A\sharp B$ denotes the  geometric mean between $A$ and $B$ defined by
\begin{equation}
\label{eq: geometric mean}
A\sharp B=A^{1/2} (A^{-1/2}BA^{-1/2})^{1/2}A^{1/2},
\end{equation}
which is symmetric in $A$ and $B$. According to \cite[Proof of
Theorem 8, Chapter 10]{Lax2007} $X\mapsto
-\ln\det(X)$ is strictly convex. Using Jacobi's formula for the
derivative of the determinant and the chain rule, we get
\begin{equation*}
Df_2(X)(Y)=-\frac{d}{dt}\ln\det(X+\varepsilon Y)\Big\vert_{t=0}=-\frac{1}{\det X}\cdot\det X \cdot\tr (X^{-1}Y)=-\tr (X^{-1}Y).
\end{equation*}
It follows that $X\mapsto f(X)$ is strictly convex. Furthermore, we have
\begin{equation*}
Df(X)(Y)=\tr\Big(I-\gamma X^{-1}-\sum_{i=1}^n \lambda_i (A_i\sharp X^{-1})\Big) Y.
\end{equation*}
From this we deduce that
\begin{equation}
\label{eq: nabla f} \nabla f(X)=I-\gamma X^{-1}-\sum_{i=1}^n
\lambda_i (A_i\sharp X^{-1}),
\end{equation}
where the gradient is with respect to the Frobenius inner product. Hence $\nabla f(X)=0$ if and only if
\begin{equation*}
I-\gamma X^{-1}=\sum_{i=1}^n \lambda_i (A_i\# X^{-1}).
\end{equation*}
Using the definition \eqref{eq: geometric mean} of the geometric mean, the above equation can be written as
\begin{equation*}
 X-\gamma I=\sum_{i=1}^n \lambda_i (X^\frac{1}{2} A_i
 X^{1/2})^\frac{1}{2},
\end{equation*}
which is equation \eqref{eq: optimal equation Gaussian}.  By Lemma \ref{lem: existence Gaussian} this equation has a positive definite solution. This together with the strict convexity of $f$ imply that $f$ has a unique minimizer which is a Gaussian measure $\N(0,X)$ where $X$ solves \eqref{eq: optimal equation Gaussian}. In the one dimensional case this equation reads
$$
X-\gamma=\sqrt{X}\sum_{i=1}^n\lambda_i\sqrt{a_i},
$$
which results in
$$
X=\frac{\left[\sum_{i=1}^n\lambda_i a_i^\frac{1}{2}+\Big(\big(\sum_{i=1}^n\lambda_i a_i^\frac{1}{2}\big)^2+4\gamma\Big)^\frac{1}{2}\right]^2}{4}.
$$
This completes the proof of the theorem.
\end{proof}
\begin{remark}[The case of non-zero mean distributions]
\label{rem: non-zero gaussians}
Assume that $\{\mu_i\}$ are Gaussian distributions with means $\{m_i\}$ and covariance matrices $\{A_i\}$, that is $\mu_i\sim \mathcal{N}(m_i,A_i)$. Using the following formulas of the Wasserstein distances 
$$
W_2^2(\mu_i,\mu)=\|m-m_i\|^2+\tr X+\tr A_i-2\tr\Big(A_i^\frac{1}{2}XA_i^\frac{1}{2}\Big)^\frac{1}{2},
$$
and the formula of the entropy functional \eqref{eq: entropy of Gaussian measures} (noting that the entropy of a normal distribution is independent of its mean), we deduce that the minimizer $\mu\sim \mathcal{N}(m,X)$ where the mean $m$ is given by
$$
m=\sum_{i=1}^n\lambda_i m_i,
$$
and the covariance matrix $X$ satisfies the nonlinear matrix equation \eqref{eq: optimal equation Gaussian}. The above statement about the mean is also true for the case of $q$-Gaussian measures and $\varphi$-exponential measures in the subsequent sections.
\end{remark}
\section{Regularization of barycenters for $q$-Gaussian measures}
\label{sec: q-Gaussians}
In this section we study the following regularization of barycenters in the space of $q$-Gaussian measures
\begin{equation}
\label{eq: penalization q-Gaussian}
\min_{\mu\in\P_2(\R^d)}\sum_{i=1}^n\frac{1}{2}\lambda_i W_2^2(\mu,\mu_i)+\gamma F_q(\mu),
\end{equation} where $\mu_i\sim G_q(0,A_i)$ ($i=1,\ldots, n$), $F_q$ is the Tsallis entropy for a probability measure $\mu=\mu(x)dx$ on $\R^d$ defined by
\begin{equation}
\label{eq: Tsallis entropy}
F_q(\mu):=\int_{\R^d} \mu(x) \log_q \mu(x)\,dx.
\end{equation}
According to Proposition \ref{prop: invariant}, we only need to seek for the minimizer $\mu$ among $q$-Gaussian measures with mean zero, that is $\mu\sim G_q(0,X)$ for some covariance matrix $X$.
\begin{theorem}
\label{thm: q-Gaussian}
Assume that $\mu_i\sim G_q(0,A_i)$. Suppose that $\alpha I\leq A_i \leq \beta I$ for all $i=1,\ldots, n$.
The regularization of barycenters problem \eqref{eq: penalization q-Gaussian} has a unique solution $\mu\sim G_q(0,X)$ for all $\gamma\geq 0$ if either $0< q \leq 1$ or $1<q\leq 1+\frac{2\alpha^2}{d \beta^2}$ and for $\gamma$ sufficiently small if $1+\frac{2\alpha^2}{d \beta^2}<q<\frac{d+4}{d+2}$. The covariance matrix $X$ solves the following nonlinear matrix equation
\begin{equation}
\label{eq: optimal equation q-Gaussian}
X-\gamma m(q,d)(\det X)^{\frac{q-1}{2}} I=\sum_{i=1}^n \lambda_i \Big(X^\frac{1}{2}A_i X^\frac{1}{2}\Big)^\frac{1}{2},
\end{equation}
where $m(q,d)$ is defined by
\begin{equation*}
m(q,d):=\frac{2(2-q)C_0(q,d)^{1-q}}{2+(d+2)(1-q)}.
\end{equation*}
\end{theorem}
The following proposition shows that equation \eqref{eq: optimal equation q-Gaussian} possesses a positive definite solution.
\begin{proposition}
\label{prop: existence solution q Gaussians}
 Equation \eqref{eq: optimal equation q-Gaussian} has a positive definite solution.
\end{proposition}
\begin{proof} Similarly as the proof of Lemma \ref{lem: existence Gaussian} we will also apply Brouwer's fixed point theorem. We will show that
$$
\psi(X):=\sum_{i=1}^n \lambda_i(X^{1/2}A_iX^{1/2})^{1/2}+\gamma m(q,d)(\det X)^{\frac{q-1}{2}} I
$$
has a fixed point which is a positive definite matrix. Due to the appearance of the second term on the left-hand side of \eqref{eq: optimal equation q-Gaussian} the proof of this proposition is significantly involved than that of Lemma \ref{lem: existence Gaussian}. Suppose that $\alpha_0 I\leq A_i \leq \beta_0 I$ for all $i=1,\ldots, n$. Then similarly as in the proof of Lemma \ref{lem: existence Gaussian}, for $\alpha_* I\leq X\leq \beta_* I$ (with $\alpha_*, \beta_*$ chosen later), we have
$$
\sqrt{\alpha_0}\sqrt{\alpha_*} I\leq \sqrt{\alpha_0}X^{1/2}\leq
(X^{1/2}A_iX^{1/2})^{1/2}\leq \sqrt{\beta_0}X^{1/2}\leq
\sqrt{\beta_0}\sqrt{\beta_*}I,\quad i=1,\ldots, n,
$$
so that
$$
\sqrt{\alpha_0}\sqrt{\alpha_*} I\leq (X^{1/2}A_iX^{1/2})^{1/2}\leq \sqrt{\beta_0}\sqrt{\beta_*}I.
$$
Multiplying this inequality with $\lambda_i$ then adding them together, noting that $\sum\lambda_i=1$, we obtain
\begin{equation*}
\sqrt{\alpha_0}\sqrt{\alpha_*} I\leq \sum_{i=1}^n \lambda_i(X^{1/2}A_iX^{1/2})^{1/2}\leq \sqrt{\beta_0}\sqrt{\beta_*}I,
\end{equation*}
from which it follows that
\begin{multline}
\label{eq: estimate1}
\sqrt{\alpha_0}\sqrt{\alpha_*} I+\gamma m(q,d)(\det X)^{\frac{q-1}{2}} I\leq \sum_{i=1}^n \lambda_i(X^{1/2}A_iX^{1/2})^{1/2}+\gamma m(q,d)(\det X)^{\frac{q-1}{2}} I\\\leq \sqrt{\beta_0}\sqrt{\beta_*}I+\gamma m(q,d)(\det X)^{\frac{q-1}{2}} I.
\end{multline}
To continue we consider two cases.
\vskip 1pc

Case 1: $1<q<\frac{d+4}{d+2}$.
It follows from \eqref{eq: estimate1} that
\begin{multline}
\label{eq: estimate2}
\sqrt{\alpha_0}\sqrt{\alpha_*} I+\gamma m(q,d)\alpha_*^{\frac{d(q-1)}{2}} I\leq
\sqrt{\alpha_0}\sqrt{\alpha_*} I+\gamma m(q,d)(\det X)^{\frac{q-1}{2}} I\\
\leq \gamma m(q,d)(\det X)^{\frac{q-1}{2}} I+\sum_{i=1}^n \lambda_i(X^{1/2}A_iX^{1/2})^{1/2}\\\leq \sqrt{\beta_0}\sqrt{\beta_*}I+\gamma m(q,d)(\det X)^{\frac{q-1}{2}} I\leq \sqrt{\beta_0}\sqrt{\beta_*}I+\gamma m(q,d)\beta_*^{\frac{d(q-1)}{2}} I.
\end{multline}
Since $1<q<\frac{d+4}{d+2}$, we have $0<(q-1)d<\frac{2d}{d+2}<2$.
\vskip 1pc
Case 1.1: $d(q-1)\leq 1$. Consider the following equation
$$
g_1(t):=t^{1-\frac{q(d-1)}{2}}-\sqrt{\alpha_0}t^{\frac{1-d(q-1)}{2}}-\gamma m(q,d)=0.
$$
We have $\lim\limits_{t\rightarrow 0}g_1(t)=-\gamma m(q,d)<0$ and $\lim\limits_{t\rightarrow +\infty}g_1(t)=+\infty$. Since $g_1$ is continuous, it follows that there exists $\alpha_*\in (0,\infty)$ such that $g_1(\alpha_*)=0$, that is
$$
\alpha_*^{1-\frac{q(d-1)}{2}}=\sqrt{\alpha_0}\alpha*^{\frac{1-d(q-1)}{2}}+\gamma m(q,d), \quad\text{i.e.,}\quad \alpha_*=\sqrt{\alpha_0}\sqrt{\alpha_*}+\gamma m(q,d)\alpha_*^{\frac{d(q-1)}{2}}.
$$
Similarly by considering the function $g_2(t):=t^{1-\frac{q(d-1)}{2}}-\sqrt{\beta_0}t^{\frac{1-d(q-1)}{2}}-\gamma m(q,d)$, we deduce that there exists $\beta_*\in (0,\infty)$ such that
$$
\beta_*=\sqrt{\beta_0}\sqrt{\beta_*}+\gamma m(q,d)\beta_*^{\frac{d(q-1)}{2}}.
$$
\vskip 1pc
Case 1.2: $d(q-1)>1$. Using the same argument as in the previous case for
$$
g_3(t)=t^{1/2}-\sqrt{\alpha_0}-\gamma m(q,d)t^{\frac{d(q-1)-1}{2}}~\text{and}~g_4(t)=t^{1/2}-\sqrt{\beta_0}-\gamma m(q,d)t^{\frac{d(q-1)-1}{2}}
$$
we can show that there exist $\alpha_*, \beta_*\in (0,\infty)$ such that
$$
\alpha_*=\sqrt{\alpha_0}\sqrt{\alpha_*}+\gamma m(q,d)\alpha_*^{\frac{d(q-1)}{2}}~\text{and}~\beta_*=\sqrt{\beta_0}\sqrt{\beta_*}+\gamma m(q,d)\beta_*^{\frac{d(q-1)}{2}}.
$$
\vskip 1pc
Therefore in both Cases 1.1 and 1.2, there exist  $\alpha_*, \beta_*\in (0,\infty)$ such that
$$
\alpha_*=\sqrt{\alpha_0}\sqrt{\alpha_*}+\gamma m(q,d)\alpha_*^{\frac{d(q-1)}{2}}~\text{and}~\beta_*=\sqrt{\beta_0}\sqrt{\beta_*}+\gamma m(q,d)\beta_*^{\frac{d(q-1)}{2}}.
$$
Substituting these quantities into \eqref{eq: estimate2} we obtain
\begin{multline}
\alpha_* I=\sqrt{\alpha_0}\sqrt{\alpha_*} I+\gamma m(q,d)\alpha_*^{\frac{d(q-1)}{2}} I \\\leq \gamma m(q,d)(\det X)^{\frac{q-1}{2}} I+\sum_{i=1}^n \lambda_i(X^{1/2}A_iX^{1/2})^{1/2} \\
\leq \sqrt{\beta_0}\sqrt{\beta_*}I+\gamma m(q,d)\beta_*^{\frac{d(q-1)}{2}} I=\beta_* I.
\end{multline}
Thus $\alpha_* I \leq \psi(X)\leq \beta_* I$.  By Brouwer's fixed point theorem, $\psi(X)$ has a fixed point in $[\alpha_* I, \beta_* I]$ as desired.
\vskip 1pc

Case 2. $0<q<1$.
\vskip 1pc

It follows from \eqref{eq: estimate1} that
\begin{multline}
\label{eq: estimate3}
\sqrt{\alpha_0}\sqrt{\alpha_*} I+\gamma m(q,d)\beta_*^{\frac{d(q-1)}{2}} I\leq
\sqrt{\alpha_0}\sqrt{\alpha_*} I+\gamma m(q,d)(\det X)^{\frac{q-1}{2}} I\\
\leq \gamma m(q,d)(\det X)^{\frac{q-1}{2}} I+\sum_{i=1}^n \lambda_i(X^{1/2}A_iX^{1/2})^{1/2}\\\leq \sqrt{\beta_0}\sqrt{\beta_*}I+\gamma m(q,d)(\det X)^{\frac{q-1}{2}} I\leq \sqrt{\beta_0}\sqrt{\beta_*}I+\gamma m(q,d)\alpha_*^{\frac{d(q-1)}{2}} I
\end{multline}
Next we will show that following system has positive solutions $0<\alpha_*<\beta_*<\infty$:
\begin{equation}
\label{eq: system1}
\begin{cases}
\alpha_*=\sqrt{\alpha_0}\sqrt{\alpha_*}+\gamma m(q,d)\beta_*^{\frac{d(q-1)}{2}} \\
\beta_*=\sqrt{\beta_0}\sqrt{\beta_*}+\gamma m(q,d)\alpha_*^{\frac{d(q-1)}{2}}.
\end{cases}
\end{equation}
Define $f: (0,\infty)^2\rightarrow (0,\infty)^2$ by
$$
f\left(\begin{pmatrix}
x\\y
\end{pmatrix}\right)=\begin{pmatrix}
\sqrt{\alpha_0}\sqrt{x}+\gamma m(q,d)y^{\frac{d(q-1)}{2}}\\
\sqrt{\beta_0}\sqrt{y}+\gamma m(q,d)x^{\frac{d(q-1)}{2}}
\end{pmatrix}
$$

Set
\begin{align*}
&a_*=\left(\frac{\sqrt{\alpha_0}+\sqrt{\alpha_0+4\gamma m(q,d)\beta_0^{(q-1)d/2}}}{2}\right)^2,
\\& b_*=\left(\frac{\sqrt{\beta_0}+\sqrt{\beta_0+4\gamma m(q,d)\alpha_0^{(q-1)d/2}}}{2}\right)^2.
\end{align*}
Thus $a_*$ and $b_*$ satisfy
$$
a_*=\sqrt{\alpha_0}\sqrt{a_*}+\gamma m(q,d)\beta_0^{(q-1)d/2},\quad b_*=\sqrt{\beta_0}\sqrt{b_*}+\gamma m(q,d)\alpha_0^{(q-1)d/2}.
$$
We now show that $f: [\alpha_0, a_*]\times [\beta_0, b_*]\rightarrow [\alpha_0, a_*]\times [\beta_0, b_*]$.
In fact, consider $\alpha_0\leq x\leq a_*$ and $\beta_0\leq y\leq b_*$. We have
\begin{align*}
&\alpha_0\leq \sqrt{\alpha_0}\sqrt{x}\leq \sqrt{\alpha_0}\sqrt{x}+\gamma m(q,d)y^{\frac{d(q-1)}{2}}\leq \sqrt{\alpha_0}\sqrt{x}+\gamma m(q,d)\beta_0^{\frac{d(q-1)}{2}}=a_*,\\
&\beta_0\leq \sqrt{\beta_0}\sqrt{y}\leq \sqrt{\beta_0}\sqrt{y}+\gamma m(q,d)x^{\frac{d(q-1)}{2}}\leq \sqrt{\beta_0}\sqrt{y}+\gamma m(q,d)\alpha_0^{\frac{d(q-1)}{2}}=b_*.\\
\end{align*}
Thus $f((x,y)^T)\in [\alpha_0, a_*]\times [\beta_0, b_*]$. By Brouwer's fixed point theorem, $f$ has a fixed point in $[\alpha_0, a_*]\times [\beta_0, b_*]$, which means that system \eqref{eq: system1} has a positive solution $(\alpha_*,\beta_*)$. Using this solution in \eqref{eq: estimate3} we obtain
\begin{multline*}
\alpha_* I=\sqrt{\alpha_0}\sqrt{\alpha_*} I+\gamma m(q,d)\beta_*^{\frac{d(q-1)}{2}} I \\ \leq \gamma m(q,d)(\det X)^{\frac{q-1}{2}} I+\sum_{i=1}^n \lambda_i(X^{1/2}A_iX^{1/2})^{1/2}\\ \leq  \sqrt{\beta_0}\sqrt{\beta_*}I+\gamma m(q,d)\alpha_*^{\frac{d(q-1)}{2}} I=\beta_* I.
\end{multline*}
Hence by Brouwer's fixed point theorem again, $\psi$ has a fixed point in $[\alpha_* I, \beta_* I]$ as desired. This finishes the proof of the proposition.
\end{proof}
Next we will show that the functional that we wish to mimimize in \eqref{eq: penalization q-Gaussian} is strictly convex under rather general conditions. According to Propositions \ref{prop: q-Gaussian properties} and Lemma \ref{lem: Tsallis entropy of q-Gaussian} we have
\begin{align*}
W_2^2(\mu_i, \mu)&=\tr X+\tr A_i-2\tr\Big(A_i^\frac{1}{2}XA_i^\frac{1}{2}\Big)^\frac{1}{2},
\\ F_q(\mu)&=-\frac{d}{2}C_1(q,d)+\Big[1-(1-q)\frac{d}{2}C_1(q,d)\Big]\ln_{q}\frac{C_0(q,d)}{(\det U)^\frac{1}{2}}.
\end{align*}
Therefore the minimization problem \eqref{eq: penalization q-Gaussian} can be written as
\begin{equation}
\label{eq: reformulation of penal. q Gaussian}
\min_{X\in \mathbb{S}(d,\mathbb{R})_+}\frac{1}{2} g(X)
\end{equation}
where
\begin{align}
\label{eq: g1}
g(X)&=\sum_{i=1}^n\lambda_i\tr A_i+\sum_{i=1}^n \lambda_i\tr\Big(X-2(A_i^\frac{1}{2}XA_i^\frac{1}{2})^\frac{1}{2}\Big)\notag
\\&\qquad+\gamma \Big[2-(1-q)d C_1(q,d)\Big]\ln_{q}\frac{C_0(q,d)}{(\det U)^\frac{1}{2}}-\gamma d C_1(q,d)\notag
\\&=f_1(X)+\gamma \Big[2-(1-q)d C_1(q,d)\Big]\ln_{q}\frac{C_0(q,d)}{(\det U)^\frac{1}{2}}-\gamma d C_1(q,d),
\end{align}
with $f_1(X)=\sum_{i=1}^n \lambda_i\tr A_i+\sum_{i=1}^n \lambda_i\tr\Big(X-2(A_i^\frac{1}{2}XA_i^\frac{1}{2})^\frac{1}{2}\Big)$, which appeared in \eqref{eq: f}.
Note that by definition of the $q$-logarithmic function we have
\begin{equation*}
\ln_{q}\frac{C_0(q,d)}{(\det U)^\frac{1}{2}}=\frac{1}{1-q}\left[C_0(q,d)^{1-q}(\det U)^{-\frac{1-q}{2}}-1\right].
\end{equation*}
Using explicit formula of $C_1(q,d)$ we get
\begin{align*}
2-(1-q)d C_1(q,d)&=2-(1-q)d\frac{2}{2+(d+2)(1-q)}
\\&=\frac{4(2-q)}{2+(d+2)(1-q)}.
\end{align*}
Substituting these expressions into \eqref{eq: g1} we get
\begin{align}
\label{eq: g2}
g(X)&=f_1(X)+\frac{4\gamma(2-q)C_0(q,d)^{1-q}}{(2+(d+2)(1-q))(1-q)}(\det X)^{-\frac{1-q}{2}}\nonumber
\\&\qquad-\frac{4(2-q)}{(1-q)(2+(d+2)(1-q))}-\gamma d C_1(q,d).
\end{align}
The following proposition studies the convexity of $g$.
\begin{proposition}
\label{prop: convexity}
Suppose that $\alpha I\leq A_i, X,\leq \beta I$ for all $i=1,\ldots, n$. The functional $g$ given in \eqref{eq: g2} is strictly convex for all $\gamma\geq 0$ when one of the following condition holds
\begin{enumerate}
\item $0<q<1$,
\item $1<q\leq 1+\frac{2\alpha^2}{d \beta^2}$.
\end{enumerate}
In addition, if $1+\frac{2\alpha^2}{d \beta^2}<q<\frac{d+4}{d+2}$, then $g$ is strictly convex for $0\leq\gamma<\gamma_0$ where
$$
\gamma_0=\frac{1}{2}\frac{\alpha^{1/2}}{\beta^{3/2}}\frac{1}{\frac{1}{\beta^2}-\frac{(q-1)d}{2\alpha^2}}\frac{1}{m(q,d)}\frac{1}{\beta^{d(q-1)/2}}.
$$
\end{proposition}
\begin{proof} We consider two cases.

\noindent{\bf Case 1.} $1< q <\frac{d+4}{d+2}.$

Let $k(X):=\frac{4\gamma(2-q)C_0(q,d)^{1-q}}{(2+(d+2)(1-q))(1-q)}(\det X)^{\frac{q-1}{2}}$. Let $h(X):= (\det X)^\frac{q-1}{2}$. Similarly as in the proof of Theorem \ref{thm: Gaussian}, using again Jacobi's formula for the derivative of the determinant and the chain rule, we get
\begin{equation*}
Dh(X)(Y)=\frac{q-1}{2}(\det X)^{\frac{q-3}{2}} \cdot \det(X)\cdot\tr (X^{-1} Y)=\frac{q-1}{2}(\det X)^{\frac{q-1}{2}}\tr (X^{-1} Y).
\end{equation*}
Therefore, using the definition of $m(q,d)$, we have
\begin{equation}
\label{eq: derivative k}
\nabla k(X) = -\gamma m(q,d)(\det X)^{\frac{q-1}{2}}X^{-1}=-\gamma m(q,d) h(X)X^{-1}.
\end{equation}
In the computations below the linear operator $P(X)$ is defined to be $P(X)Y = XYX$. This operator is called the quadratic representation in the literature. By the Leibniz rule, we get
\begin{align*}
\nabla^2 k(X)(H) &= D(\nabla k)(X)(H)\\
& = -\gamma m(q,d)[D h(X)(H) X^{-1} + h(X)(-P(X^{-1}))(H)]\\
& = -\gamma m(q,d)[\langle \nabla h(X), H\rangle X^{-1} - h(X)X^{-1}H X^{-1}]\\
& = -\gamma m(q,d)\left[\left\langle \frac{q-1}{2}(\det X)^{\frac{q-1}{2}}X^{-1},  H\right\rangle X^{-1}- (\det X)^{\frac{q-1}{2}}X^{-1}H X^{-1}\right]\\
& = -\gamma m(q,d)(\det X)^{\frac{q-1}{2}}\left[\left\langle \frac{q-1}{2}X^{-1}, H\right\rangle X^{-1}- X^{-1}H X^{-1}\right].
\end{align*}
Thus
\begin{align}
\langle\nabla^2 k(X)(H), H\rangle &=-\gamma m(q,d)(\det X)^{\frac{q-1}{2}}\left[\frac{q-1}{2}\left\langle X^{-1}, H\right\rangle^2-\langle X^{-1}H, X^{-1}H\rangle\right]\\
&=-\gamma m(q,d)(\det X)^{\frac{q-1}{2}}\left[\frac{q-1}{2}\tr^2 (X^{-1}H)-\| X^{-1}H\|^2\right].
\end{align}
\noindent Furthermore, according to \cite{Bhatia2018}, for $\alpha I \le A_i, X \le \beta I$, we have
$$
\langle\nabla^2 f_1 (X)(H), H\rangle\geq \frac{1}{2}\frac{\alpha^{1/2}}{\beta^{3/2}}\|H\|^2.
$$
Thus we get
\begin{align*}
\langle\nabla^2 g(X)(H), H\rangle &= \langle\nabla^2 f_1 (X)(H), H\rangle + \langle\nabla^2 k(X)(H), H\rangle\\
&\ge -\gamma m(q,d)(\det X)^{\frac{q-1}{2}}\left[\frac{q-1}{2}\tr^2 (X^{-1}H)-\| X^{-1}H\|^2\right] + \frac{1}{2}\frac{\alpha^{1/2}}{\beta^{3/2}}\|H\|^2\\
&= \gamma m(q,d)(\det X)^{\frac{q-1}{2}}\left[\langle P(X^{-1})H, H\rangle - \frac{q-1}{2}\tr^2 (X^{-1}H)\right] + \frac{1}{2}\frac{\alpha^{1/2}}{\beta^{3/2}}\|H\|^2\\
& \ge \gamma m(q,d)(\det X)^{\frac{q-1}{2}}\left[\frac{1}{\beta^2}\|H\|^2 - \frac{q-1}{2}\|X^{-1}\|^2\|H\|^2\right] + \frac{1}{2}\frac{\alpha^{1/2}}{\beta^{3/2}}\|H\|^2\\
& = \left\{\gamma m(q,d)(\det X)^{\frac{q-1}{2}}\left[\frac{1}{\beta^2} - \frac{q-1}{2}\|X^{-1}\|^2\right] + \frac{1}{2}\frac{\alpha^{1/2}}{\beta^{3/2}}\right\}\|H\|^2\\
& \ge\left\{\gamma m(q,d)(\det X)^{\frac{q-1}{2}}\left[\frac{1}{\beta^2} - \frac{q-1}{2}\frac{d}{\alpha^2}\right] + \frac{1}{2}\frac{\alpha^{1/2}}{\beta^{3/2}}\right\}\|H\|^2\\
& \ge \left\{\gamma m(q,d)(\det X)^{\frac{q-1}{2}}\left[\frac{1}{\beta^2} - \frac{q-1}{2}\frac{d}{\alpha^2}\right] + \frac{1}{2}\frac{\alpha^{1/2}}{\beta^{3/2}}\right\}\|H\|^2.
\end{align*}
From this estimate, we deduce the following cases
\begin{enumerate}[(i)]
\item If
$$
1<q\leq 1+\frac{2\alpha^2}{d \beta^2},
$$
thus $\frac{1}{\beta^2}-\frac{q-1}{2}\frac{d}{\alpha^2}>0$, which implies that the Hessian of $g$ is positive for all $\gamma$. Note that the above condition is fulfilled if $\alpha$ and $\beta$ satisfy  $ \beta^2\leq \frac{d+2}{d}\alpha^2$. In fact, we have
$$
q<1+\frac{2}{d+2}\leq 1+\frac{2\alpha^2}{d\beta^2},
$$
\item If
$$
1+\frac{2\alpha^2}{d \beta^2}<q<\frac{d+4}{d+2}.
$$
then for
$$
\gamma<\frac{1}{2}\frac{\alpha^{1/2}}{\beta^{3/2}}\frac{1}{\frac{1}{\beta^2}-\frac{(q-1)d}{2\alpha^2}}\frac{1}{m(q,d)}\frac{1}{\beta^{d(q-1)/2}}
$$
the Hessian of $g$ is positive since
\begin{multline}
\gamma<\frac{1}{2}\frac{\alpha^{1/2}}{\beta^{3/2}}\frac{1}{\frac{1}{\beta^2}-\frac{(q-1)d}{2\alpha^2}}\frac{1}{m(q,d)}\frac{1}{\beta^{d(q-1)/2}} \\ \leq \frac{1}{2}\frac{\alpha^{1/2}}{\beta^{3/2}}\frac{1}{\frac{1}{\beta^2}-\frac{(q-1)d}{2\alpha^2}}\frac{1}{m(q,d)}\frac{1}{(\det X)^{(q-1)/2}}
\end{multline}
\end{enumerate}
\vskip 1pc
\noindent{\bf Case 2.} $0< q <1.$
Similarly, we obtain
\begin{align*}
\langle\nabla^2 k(X)(H), H\rangle &=\gamma m(q,d)(\det X)^{\frac{q-1}{2}}\left[\frac{1-q}{2}\left\langle X^{-1}, H\right\rangle^2+\langle P(X^{-1})H, H\rangle\right]\\
&\ge\gamma m(q,d)(\det X)^{\frac{q-1}{2}}{\frac{1}{\lambda^2_{\max}(X)}}\|H\|^2.
\end{align*}
Hence the Hessian of $g$ is always positive definite in this case.
\end{proof}
We are now ready to proof Theorem \ref{thm: q-Gaussian}.
\begin{proof}[Proof of Theorem \ref{thm: q-Gaussian}]
Suppose that the hypothesis of the statement of Theorem \ref{thm: q-Gaussian} is satisfied, that is either (i) $0<q\leq 1$ or (ii) $1<q\leq 1+\frac{2\alpha^2}{d \beta^2}$ or (iii) $1+\frac{2\alpha^2}{d \beta^2}<q<\frac{d+4}{d+2}$. Suppose that $\gamma$ is sufficiently small in the last case; in the other cases it can be arbitrarily positive. As has been shown in the paragraph before Proposition \ref{prop: convexity}, the minimization problem \eqref{eq: penalization q-Gaussian} can be written as
\begin{equation*}
\min_{X\in\H}\frac{1}{2} g(X),
\end{equation*}
where $g(X)$ is given in \eqref{eq: g2}
$$
g(X)=f_1(X)+k(X)-\frac{4(2-q)}{(1-q)(2+(d+2)(1-q))}-\gamma d C_1(q,d).
$$
By Proposition \ref{prop: convexity}, $X\mapsto g(X)$ is strictly convex.
Now we compute the derivative of $g(X)$. We have
\begin{equation}
\label{eq: derivative 1}
\nabla g(X)= \nabla f_1(X)+ \nabla k(X),
\end{equation}
According to the proof of Theorem \ref{thm: Gaussian} we have
\begin{equation*}
\nabla f_1(X)=I-\sum_{i=1}^n \lambda_i (A_i\sharp X^{-1}).
\end{equation*}
By \eqref{eq: derivative k}, we have
$$
\nabla k(X)=-\gamma m(q,d) (\det X)^{\frac{q-1}{2}} X^{-1}
$$
Substituting these computations into \eqref{eq: derivative 1} we obtain
\begin{equation*}
\nabla g(X)=\Big(I-\sum_{i=1}^n \lambda_i (A_i\sharp X^{-1})\Big)-\gamma m(q,d)(\det X)^{\frac{q-1}{2}}X^{-1}.
\end{equation*}
Thus $\nabla g(X)=0$ if and only if
$$
I-\gamma m(q,d)(\det X)^{\frac{q-1}{2}}X^{-1}=\sum_{i=1}^n \lambda_i (A_i\sharp X^{-1}),
$$
which, by using the definition of the geometric mean \eqref{eq: geometric mean}, is equivalent to
$$
X-\gamma m(q,d)(\det X)^{\frac{q-1}{2}} I=\sum_{i=1}^n \lambda_i \Big(X^\frac{1}{2}A_i X^\frac{1}{2}\Big)^\frac{1}{2}.
$$
This is precisely equation \eqref{eq: optimal equation q-Gaussian}. By Proposition \ref{prop: existence solution q Gaussians}, it has a positive definite solution. This, together with the strictly convexity of $g$, guarantees the existence and uniqueness of a minimizer of $g$. We complete the proof of the theorem.
\end{proof}

\section{Barycenters for $\varphi$-exponential measures}
\label{sec: vaphi expo measures}
In this section we consider the following barycenter problem in the space of $\varphi$-exponential measures:
\begin{equation}
\label{eq: baryceter varphi}
\min_{\mu\in\P_2(\R^d)}\sum_{i=1}^n\frac{\lambda_i}{2}W_2^2(\mu,\mu_i).
\end{equation}
In contrast to the Gaussian and q-Gaussian measures, we are not aware of an explicit formulation for the entropy for a general $\varphi$-exponential measure. Therefore, in the above formulation we do not include the regularization term. The main result of this section is the following theorem that states that the equation determining the barycenter for $\varphi$-exponential measures is the same as that of for Gaussian-measures.
\begin{theorem}
\label{thm: varphi}
Let $\varphi\in \O(2/(d+2))$ with $d\geq 2$. Assume that $\mu_i\sim G_{\varphi}(0,A_i)$. The non-regularization of barycenters problem \eqref{eq: penalization} has a unique solution $\mu\sim G_\varphi(0,X)$ where the covariance matrix $X$ solves the following nonlinear matrix equation
\begin{equation}
\label{eq: optimal equation varphi}
X=\sum_{i=1}^n \lambda_i (X^\frac{1}{2} A_i X^{1/2})^\frac{1}{2}.
\end{equation}
In particular, for $n=2$, $X$ is given explicitly by
\begin{equation}
\label{eq: optimal n=2}
X=\lambda_1^2 A_1+ \lambda_2^2 A_2+\lambda_1 \lambda_2 \Big[(A_1A_2)^\frac{1}{2}+(A_2A_1)^\frac{1}{2}\Big].
\end{equation}
\end{theorem}
\begin{proof}
This theorem is a direct consequence of Proposition \ref{prop: q-Gaussian properties} and \cite[Theorem 6.1]{Agueh2011} or \cite[Theorem 8]{Bhatia2018}. In fact, similarly as in the proof of \ref{thm: q-Gaussian}, by using \eqref{eq: W2 formula} we can write \eqref{eq: baryceter varphi} as
$$
\min_{X\in \mathbb{S}(d,\R^d)_+}\frac{1}{2}f_1(X)
$$
where $f_1(x)=\sum_{i=1}^n \lambda_i\tr(A_i)+\sum_{i=1}^n\lambda_i\tr\Big(X-2(A_i^\frac{1}{2}XA_i^\frac{1}{2})\Big)^\frac{1}{2}$. Then the statement can be proved exactly as  \cite[Theorem 6.1]{Agueh2011} or \cite[Theorem 8]{Bhatia2018}, see also computations in the proof of Theorems \ref{thm: Gaussian} and \ref{thm: q-Gaussian} when $\gamma=0$.
Explicit formula \eqref{eq: optimal n=2} for the minimizer for the case $n=2$ is given in \cite[Eq. (63)]{Bhatia2018}.
\end{proof}

\section{Gradient projection method}
\label{sec: GPM}
In this section, we describe a gradient projection method for the computation of the minimizer to the regularization problems (\ref{eq: penlization2}) and (\ref{eq: reformulation of penal. q Gaussian}), and analyze its convergence properties. 

First, we formally describe the algorithmic procedure for the gradient projection method (GPM) below.
\begin{algorithm}\label{A:alg1}
	\caption{GPM}
	\begin{algorithmic}
		\STATE Choose $X^0\in [\hat{\alpha}I,\hat{\beta}I]$.  Initialize $k=0$. Update $X^{(k+1)}$ from
		$X^{(k)}$ by the following template:
		\begin{description}
			\item[Step 1.] Find $\bar{X}^{(k)}=[X^{(k)}-\nabla \psi(X^{(k)})]^+$,
			\item[Step 2.] Select a stepsize $t^{(k)}$,
			\item[Step 3.] $X^{(k+1)}=X^{(k)}+t^{(k)}(\bar{X}^{(k)}-X^{(k)})$.
		\end{description}
		
		Here $[\cdot]^+$ denotes the projection on the set $[\hat{\alpha}I,\hat{\beta}I]$.
	\end{algorithmic}
\end{algorithm}

The stepsize is selected by Armijo rule along the feasible direction \cite{Ber99}. It is described below. 

\vspace{4mm} \centerline{\fbox{\parbox{5in}{ Let $t^{(k)}$ be the largest element of
			$\{\xi^j\}_{j=0,1,...}$ satisfying
			\begin{equation}\label{armijo}
			\psi(X^{(k)}+t^{(k)} D^{(k)}) \le \psi(X^{(k)}) - \sigma t^{(k)}
			\inprod{\nabla \psi(X^{(k)})}{D^{(k)}},
			\end{equation}
			where $0<\xi<1$, $0<\sigma<1$, and
			$D^{(k)}=\bar{X}^{(k)}-X^{(k)}$. }}}
Note that $\psi=f$ for the regularization problem (\ref{eq: penlization2}) and $\psi=g$ for the regularization problem (\ref{eq: reformulation of penal. q Gaussian}). The projection of the matrix $S\in {\mathcal S}^d$, where ${\mathcal S}^d$ is the set of $d\times d$ symmetric matrices, onto the set $[\hat{\alpha}I,\hat{\beta}I]$ is to find the solution of the following minimization problem
$$
\min_{X\in [\hat{\alpha}I,\hat{\beta}I]}\;\norm{X-S}_F.
$$
The solution of the above problem is
$$
[S]^+=U{\rm Diag}(\min(\max(\hat{\alpha},\lambda_1),\hat{\beta}),\dots,\min(\max(\hat{\alpha},\lambda_d),\hat{\beta})U^T,
$$
where $\lambda_1\ge\cdots\ge\lambda_d$ are the eigenvalues of $S$ and $U$ is a corresponding orthogonal matrix of eigenvalues of $S$. 

Now, we establish the global convergence of GPM. For the proof, we refer to
\cite[Proposition 2.3.1]{Ber99}.

\begin{theorem}\label{thm1}
	Let $\{X^{(k)}\}$ be the sequence generated by GPM with $t^{(k)}$ chosen by Armijo rule
	along the feasible direction. Then
	every limit point of $\{X^{(k)}\}$ is stationary.
\end{theorem}

In the following subsections, we show the Lipschitz continuity of the gradient function of the regularization problems. In this case, we can use a constant stepsize for
the gradient projection method. That is, $t^{(k)}=\frac{1}{L}$ where $L$ is a Lipschitz constant. Then we have
\begin{equation}\label{con}
X^{(k+1)}=X^{(k)}+\frac{1}{L}(\bar{X}^{(k)}-X^{(k)}).
\end{equation}

\subsection{Regularization of barycenters for Gaussian measures}
We recall that the unique minimizer of the minimization problem \eqref{eq: penalization Gaussian} in the space of Gaussian measures satisfies the following nonlinear matrix equation $\nabla f(X)=0$ where
\begin{equation*}
\nabla f(X)=I-\sum_{i=  1}^n\lambda_i (A_i\sharp X^{-1})-\gamma X^{-1}=:F_1(X)-\gamma F_2 (X).
\end{equation*}
We establish the following theorem for the Lipschitz continuity of the gradient function.
\begin{theorem}
\label{thm: Lipschitz continuity}
Suppose that $A_i\in [\alpha I, \beta I]$ for all $i=1,\ldots, n$. Then for $\alpha I\leq X\neq Y\leq \beta I$ we have
\begin{equation*}
\frac{\|\nabla f(X)-\nabla f(Y)\|_F}{\|X-Y\|_F}\leq \frac{\beta^2}{2\alpha^3}+\frac{\gamma}{\alpha^2}.
\end{equation*}
\end{theorem}
\begin{proof}
According to \cite[Proof of Theorem 3.1]{KumYun2018} we have
\begin{align*}
&\frac{\|F_1(X)-F_1(Y)\|_F}{\|X-Y\|_F}\leq \frac{\beta^2}{2\alpha^3}~~\text{and}~~\frac{\|F_2(X)-F_2(Y)\|_F}{\|X-Y\|_F}\leq \frac{1}{\alpha^2}.
\end{align*}
Therefore we get
\begin{align*}
\frac{\|\nabla f(X)-\nabla f(Y)\|_F}{\|X-Y\|_F}&\leq \frac{\|F_1(X)-F_1(Y)\|_F+\gamma \|F_2(X)-F_2(Y)\|_F}{\|X-Y\|_F} 
\\&\leq\frac{\beta^2}{2\alpha^3}+\frac{\gamma}{\alpha^2}.
\end{align*}
\end{proof}

\subsection{Regularization of barycenters for $q$-Gaussian measures}
We recall that the unique minimizer of the minimization problem \eqref{eq: penalization q-Gaussian} in the space of $q$-Gaussian measures solves the nonlinear matrix equation $\nabla g(X)$=0 where
\begin{equation}
\label{eq: nabla g(X)}
\nabla g(X)=\Big(I-\sum_{i=1}^n \lambda_i (A_i\sharp X^{-1})\Big)-\gamma m(q,d)(\det X)^{\frac{q-1}{2}}X^{-1}=: F_1(X)-\gamma m(q,d)\tilde h(X),
\end{equation}
where $F_1(X)=\Big(I-\sum_{i=1}^n \lambda_i (A_i\sharp X^{-1})\Big)$ as in the previous section and $\tilde h(X) =(\det X)^{\frac{q-1}{2}}X^{-1}=h(X)X^{-1}$.
The following main theorem of this section proves the Lipschitz continuity of $\nabla g$.
\begin{theorem}
\label{thm: Lipschitz continuity q-Gaussians}
Suppose that $A_i\in [\alpha I, \beta I]$ for all $i=1,\ldots, n$. Then for  $\alpha I\leq X\neq Y\leq \beta     I$, we have
\begin{equation*}
\frac{\|\nabla g(X)-\nabla g(Y)\|_F}{\|X-Y\|_F}\leq \begin{cases}
\frac{\beta^2}{2\alpha^3}+\frac{\gamma}{\alpha^2}+\frac{\gamma m(q,d)}{\alpha^2}\cdot \beta^{\frac{q-1}{2}d} \bigg(1+ \frac{q-1}{2}d\bigg),\quad \text{if}~~1< q <\frac{d+4}{d+2},\\
\\\frac{\beta^2}{2\alpha^3}+\frac{\gamma}{\alpha^2}+\gamma m(q,d)\alpha^{-2 +\frac{q-1}{2}d} \bigg(1+ \frac{1-q}{2} d\bigg),\quad \text{if}~~0< q <1.
\end{cases}
\end{equation*}
\end{theorem}
\begin{proof}
Let $\alpha I\le X,\, Y \le \beta I$. According to the proof of Theorem \ref{thm: Lipschitz continuity}, we have
\begin{equation}
\label{eq: F1}
\frac{\|F_1(X)-F_1(Y)\|_F}{\|X-Y\|_F}\leq \frac{\beta^2}{2\alpha^3}+\frac{\gamma}{\alpha^2}.
\end{equation}
It remains to study the Lipschitz continuity of $\tilde h(X) =(\det X)^{\frac{q-1}{2}}X^{-1}=h(X)X^{-1}$.
\vskip 1pc
\noindent{\bf Case 1.} $1< q <\frac{d+4}{d+2}.$
First, we have
\begin{align*}
|h(X) - h(Y)| &= \big|\exp (\ln (\det X)^{\frac{q-1}{2}})- \exp (\ln (\det Y)^{\frac{q-1}{2}})\big| \\
&= e^{\theta}\, \big|\ln (\det X)^{\frac{q-1}{2}}-\ln (\det Y)^{\frac{q-1}{2}}\big|\\
&\le \beta^{\frac{q-1}{2}d}\, \big|\ln (\det X)^{\frac{q-1}{2}}-\ln (\det Y)^{\frac{q-1}{2}}\big|\\
& =   \frac{q-1}{2}\cdot\beta^{\frac{q-1}{2}d}\,|\ln \det X-\ln \det Y|\\
& \le  \frac{q-1}{2}\cdot\beta^{\frac{q-1}{2}d} \bigg(\max_{\alpha I\le X \le \beta I} \|X^{-1}\|\bigg) \|X-Y\|\\
& \le \frac{q-1}{2}\cdot\beta^{\frac{q-1}{2}d}\cdot \frac{\sqrt{d}}{\alpha}  \|X-Y\|
\end{align*}
where $ \ln \alpha^{\frac{q-1}{2}d}\le \theta \le \ln \beta^{\frac{q-1}{2}d}$ because $\ln \alpha^{\frac{q-1}{2}d}\le \ln (\det X)^{\frac{q-1}{2}} \le \ln \beta^{\frac{q-1}{2}d}$.
The second equality and inequality are derived from the mean value theorem. Moreover, we get
\begin{align}
\label{eq: tilde h1}
\| \tilde h(X) - \tilde h(Y)\| & = \| h(X) (X^{-1} -Y^{-1}) + (h(X) - h(Y)) Y^{-1}\|\notag\\
&\le h(X) \|X^{-1} -Y^{-1}\| + |h(X) - h(Y)|\, \|Y^{-1}\|\notag \\
&\le \bigg(\max_{\alpha I\le X \le \beta I} h(X) \bigg)\cdot \frac{1}{\alpha^2} \| X - Y\|\notag\\
&+ \bigg(\max_{\alpha I\le Y \le \beta I} \|Y^{-1}\|\bigg)\cdot \frac{q-1}{2}\cdot\beta^{\frac{q-1}{2}d}\cdot \frac{\sqrt{d}}{\alpha}  \|X-Y\| \notag \\
&= \bigg(\beta^{\frac{q-1}{2}d}\cdot \frac{1}{\alpha^2} + \frac{\sqrt{d}}{\alpha}\cdot\frac{q-1}{2}\cdot\beta^{\frac{q-1}{2}d}\cdot \frac{\sqrt{d}}{\alpha}\bigg)  \|X-Y\|\notag \\
& = \frac{1}{\alpha^2}\cdot \beta^{\frac{q-1}{2}d} \bigg(1+ \frac{q-1}{2}d\bigg)  \|X-Y\|
\end{align}
where the second inequality comes from \cite[Proof of Theorem 3.1]{KumYun2018}.
\vskip 1pc
\noindent{\bf Case 2.} $0< q <1.$
Similarly, we obtain
$$
|h(X) - h(Y)| \le \frac{1-q}{2}\cdot\alpha^{\frac{q-1}{2}d}\cdot \frac{\sqrt{d}}{\alpha}  \|X-Y\|.
$$
Hence
\begin{align}
\label{eq: tilde h2}
\| \tilde h(X) - \tilde h(Y)\| & \le \bigg(\max_{\alpha I\le X \le \beta I} h(X) \bigg)\cdot \frac{1}{\alpha^2} \| X - Y\|\notag\\
&+ \bigg(\max_{\alpha I\le Y \le \beta I} \|Y^{-1}\|\bigg)\cdot \frac{1-q}{2}\cdot\alpha^{\frac{q-1}{2}d}\cdot \frac{\sqrt{d}}{\alpha}  \|X-Y\| \notag \\
& = \bigg(\alpha^{\frac{q-1}{2}d}\cdot \frac{1}{\alpha^2} + \frac{\sqrt{d}}{\alpha}\cdot\frac{1-q}{2}\cdot\alpha^{\frac{q-1}{2}d}\cdot \frac{\sqrt{d}}{\alpha}\bigg)  \|X-Y\| \notag \\
& = \alpha^{-2 +\frac{q-1}{2}d} \bigg(1+ \frac{1-q}{2} d\bigg) \|X-Y\|.
\end{align}
Substituting the estimates \eqref{eq: F1}, \eqref{eq: tilde h1} and \eqref{eq: tilde h2} back into \eqref{eq: nabla g(X)} we obtain the desired inequality.
\end{proof}

\section{Numerical Experiments}
\label{sec: numerical experiments}
In this section, we numerically observe how the solution is effected as $q\to 1$. To see this, we report numerical results of a gradient projection method applied for the regularization of barycenters for $q$-Gaussian measures on $n$ randomly generated matrices of the size $d\times d$. The random matrices we use for our test are generated by matlab code as follows:
$$
\begin{array}{rl}
for & i=1:n \\
& [Q,~]=qr(randn(d)); \\
& A_i=Q*diag({\rm eiglb}+{\rm eigub}*rand(d,1))*Q';
\end{array}
$$
The eigenvalues of generated matrices are randomly distributed in the interval $[{\rm eiglb},{\rm eiglb}+{\rm eigub}]$. In our experiments, we set $n=100$, $d=10$ if $q<1$ and $n=50$, $d=5$ if $q>1$. And we set ${\rm eiglb}=0.1$ and ${\rm eigub}=9.9$.

We set $\xi=0.5$, $\sigma=0.1$, $\hat{\alpha}=10^{-5}$, $\hat{\beta}=10^5$, $\lambda_i=1/n,\;i=1,\dots,n$ for GPM in our experiment. 
All runs are performed on a Laptop with Intel Core i7-10510U CPU (2.30GHz)
and 16GB Memory, running 64-bit windows 10 and MATLAB (Version 9.8).
Throughout the experiments, we choose the initial iterate to be $X^0 = I$ and stop the algorithm when $\norm{D^{(k)}}_F\le 10^{-8}$. 

\begin{table}[htb!]
	\caption{Test results of the value $\|X_{0.5}-X_q\|_F$ where $X_{0.5}$ is the final estimated solution of the model (\ref{eq: reformulation of penal. q Gaussian}) with $q=0.5$ and $X_q$ is that with various given $q$ less than $1$ on $5$ random data sets.} \label{t1}
	\begin{center}
			\begin{tabular}{| c | c| c| c| c| c|} \hline
				\mc{1}{|c|}{q} &\mc{5}{c|}{difference when $\gamma=1$}\\[3pt] \hline
				0.6 & 0.00502 & 0.00503 &0.00521  & 0.00481  & 0.00497 \\[3pt]
				0.7 & 0.04672 & 0.04679 & 0.04761 & 0.04572  & 0.04649 \\[3pt]
				0.8 & 0.39716 & 0.39688 & 0.39667 & 0.39682 & 0.39653 \\[3pt]
				0.9 & 3.14528 & 3.13602 & 3.07209 & 3.21587 & 3.15235  \\[3pt]
				0.99 & 10.24065 & 10.19128 & 9.81731 & 10.67508 & 10.29661 \\[3pt] \hline
				\mc{1}{|c|}{q} &\mc{5}{c|}{difference when $\gamma=0.1$}\\[3pt] \hline
				0.6 & 0.000501 & 0.000503 & 0.000520 & 0.000481 & 0.000497 \\[3pt]
				0.7 & 0.00466 & 0.00467 & 0.00475  & 0.00456 & 0.00464 \\[3pt]
				0.8 & 0.03947  & 0.03944 & 0.03941 & 0.03944 & 0.03940 \\[3pt]
				0.9 & 0.33191 & 0.33097 & 0.32457 & 0.33896 & 0.33259 \\[3pt]
				0.99 & 2.08373 & 2.07390 & 1.99968 & 2.16978 & 2.09474 \\[3pt] \hline
				\mc{1}{|c|}{q} &\mc{5}{c|}{difference when $\gamma=0.01$}\\[3pt] \hline
				0.6 & 0.0000501 & 0.0000502 & 0.0000519 & 0.0000481 & 0.0000497 \\[3pt]
				0.7 & 0.00047 & 0.00047 & 0.00047 & 0.00046 & 0.00046 \\[3pt]
				0.8 & 0.00394 & 0.00394 & 0.00394 & 0.00394 & 0.00394 \\[3pt]
				0.9 & 0.03337 & 0.03327 & 0.03263 & 0.03407 & 0.03343 \\[3pt]
				0.99 & 0.22964 & 0.22856 & 0.22042 & 0.23908 & 0.23085 \\[3pt] \hline
		\end{tabular}
	\end{center}
\end{table}

\begin{table}[htb!]
	\caption{Test results of the value $\|X_{1.25}-X_q\|_F$ where $X_{1.25}$ is the final estimated solution of the model (\ref{eq: reformulation of penal. q Gaussian}) with $q=1.25$ and $X_q$ is that with various given $q$ greater than $1$ on $5$ random data sets.} \label{t2}
	\begin{center}
		\begin{tabular}{| c | c| c| c| c| c|} \hline
			\mc{1}{|c|}{q} &\mc{5}{c|}{difference when $\gamma=0.1$}\\[3pt] \hline
			1.2 & 1.11865 & 1.05088 & 1.01739 & 1.16269 & 1.13214 \\[3pt]
			1.1 & 3.54827 & 3.29863 & 3.17136  & 3.71393 & 3.60257 \\[3pt]
			1.01 & 5.49360 & 5.08760 & 4.87969 & 5.76457 & 5.58347 \\[3pt] \hline
			\mc{1}{|c|}{q} &\mc{5}{c|}{difference when $\gamma=0.01$}\\[3pt] \hline
			1.2 & 1.05337 & 0.95910 & 0.91030 & 1.11727 & 1.07516 \\[3pt]
			1.1 & 2.09875 & 1.90800 & 1.80934 & 2.22836 & 2.14300 \\[3pt]
			1.01 & 2.46054 & 2.23918 & 2.12473 & 2.61087 & 2.51182 \\[3pt] \hline
		\end{tabular}
	\end{center}
\end{table}

We report in Table \ref{t1} our numerical results, showing the Frobenius norm of the difference between the final estimated solution of the model (\ref{eq: reformulation of penal. q Gaussian}) with $q=0.5$ and that with various given $q$ less than $1$. In Table \ref{t2}, the difference between the final estimated solution of the model (\ref{eq: reformulation of penal. q Gaussian}) with $q=1.25$ and that with various given $q$ greater than $1$ is reported. From Tables \ref{t1}-\ref{t2}, we see that the difference is increasing as $q$ goes to $1$ when the regularization parameter $\gamma$ is fixed and we observe that the bigger the regularization parameter $\gamma$ is, the bigger the difference is when $q$ is fixed.

In the next experiment, we investigate stability properties for the model (\ref{eq: reformulation of penal. q Gaussian}). 
We perturb the given data, $A_i$ as follows:
$$
B_i=A_i+\epsilon I\quad i=1,\dots,n
$$

\begin{table}[htb!]
	\caption{Test results of the value $\|X_B-X_A\|_F/\epsilon$ where $X_A$ is the final estimated solution of the model (\ref{eq: reformulation of penal. q Gaussian}) with data $A_i$ and $X_B$ is that with the perturbed data $B_i$ on $5$ random data sets when $q<1$.} \label{t3}
\begin{center}
		{\scriptsize
		\begin{tabular}{| c | c| c| c| c| c| c| c| c| c| c| c| c| c| c| c|} \hline
			\mc{1}{|c|}{q} &\mc{5}{c|}{ $\gamma=1$ and $\epsilon=10^{-2}$}&\mc{5}{c|}{ $\gamma=1$ and $\epsilon=10^{-3}$}&\mc{5}{c|}{ $\gamma=1$ and $\epsilon=10^{-5}$}\\[3pt] \hline
			0.6 & 3.90 & 3.79 &	3.88 &	3.83 &	3.84 & 3.91 & 3.79 &	3.88 & 3.84 & 3.85 & 3.91 & 3.79 &	3.89 & 3.84 & 3.85\\[3pt]
			0.7 & 3.90 & 3.79 &	3.88 &	3.84 &	3.85 & 3.91 & 3.80 &	3.89 & 3.84 & 3.86 & 3.91 & 3.80 &	3.89 & 3.84 & 3.86\\[3pt]
			0.8 & 3.90 & 3.79 &	3.88 &	3.83 &	3.84 & 3.91 & 3.79 &	3.88 & 3.84 & 3.85 & 3.91 & 3.79 &	3.88 & 3.84 & 3.85\\[3pt]
			0.9 & 3.45 & 3.35 &	3.42 &	3.40 &	3.40 & 3.45 & 3.35 &	3.43 & 3.40 & 3.41 & 3.46 & 3.35 &	3.43 & 3.40 & 3.41 \\[3pt]
			0.99 & 1.19 & 1.15 & 1.18 &	1.17 &	1.17 & 1.19 & 1.15 & 1.18 &	1.17 & 1.17 & 1.19 & 1.15 & 1.18 &	1.17 & 1.17\\[3pt] \hline
			\mc{1}{|c|}{} &\mc{5}{c|}{ $\gamma=0.1$ and $\epsilon=10^{-2}$}&\mc{5}{c|}{ $\gamma=0.1$ and $\epsilon=10^{-3}$}&\mc{5}{c|}{ $\gamma=0.1$ and $\epsilon=10^{-5}$}\\[3pt] \hline
		    0.6 & 3.90 & 3.79 &	3.88 & 3.83 & 3.84 & 3.90 & 3.79 &	3.88 & 3.84 & 3.85 & 3.91 & 3.79 &	3.88 & 3.84 & 3.85\\[3pt]
		    0.7 & 3.90 & 3.79 &	3.88 & 3.83 & 3.84 & 3.91 & 3.79 &	3.88 & 3.84 & 3.85 & 3.91 & 3.79 &	3.88 & 3.84 & 3.85\\[3pt]
		    0.8 & 3.90 & 3.79 &	3.88 & 3.83 & 3.84 & 3.90 & 3.79 &	3.88 & 3.84 & 3.85 & 3.91 & 3.79 &	3.88 & 3.84 & 3.85\\[3pt]
		    0.9 & 3.85 & 3.74 &	3.83 & 3.79 & 3.80 & 3.86 & 3.75 &	3.84 & 3.79 & 3.80 & 3.86 & 3.75 &	3.84 & 3.79 & 3.81\\[3pt]
		    0.99 & 3.36 & 3.27 & 3.34 &	3.30 & 3.31 & 3.37 & 3.27 & 3.35 & 3.31 & 3.32 & 3.37 & 3.27 & 3.35 &	3.31 & 3.32\\[3pt] \hline
			\mc{1}{|c|}{} &\mc{5}{c|}{ $\gamma=0.01$ and $\epsilon=10^{-2}$}&\mc{5}{c|}{ $\gamma=0.01$ and $\epsilon=10^{-3}$}&\mc{5}{c|}{ $\gamma=0.01$ and $\epsilon=10^{-5}$}\\[3pt] \hline
			0.6 & 3.90 & 3.79 &	3.88 & 3.83 & 3.84 & 3.90 & 3.79 &	3.88 & 3.84 & 3.85 & 3.91 & 3.79 &	3.88 & 3.84 & 3.85\\[3pt]
			0.7 & 3.90 & 3.79 &	3.88 & 3.83 & 3.84 & 3.90 & 3.79 &	3.88 & 3.84 & 3.85 & 3.91 & 3.79 &	3.88 & 3.84 & 3.85\\[3pt]
			0.8 & 3.90 & 3.79 &	3.88 & 3.83 & 3.84 & 3.90 & 3.79 &	3.88 & 3.84 & 3.85 & 3.91 & 3.79 &	3.88 & 3.84 & 3.85\\[3pt]
			0.9 & 3.89 & 3.78 &	3.87 & 3.83 & 3.84 & 3.90 & 3.79 &	3.88 & 3.83 & 3.84 & 3.90 & 3.79 &	3.88 & 3.83 & 3.85 \\[3pt]
			0.99 & 3.84 & 3.73 & 3.82 &	3.77 & 3.78 & 3.85 & 3.73 & 3.82 &	3.78 & 3.79 & 3.85 & 3.74 & 3.82 &	3.78 & 3.79\\[3pt] \hline
		\end{tabular}}
	\end{center}
\end{table}

\begin{table}[htb!]
	\caption{Test results of the value $\|X_B-X_A\|_F/\epsilon$ where $X_A$ is the final estimated solution of the model (\ref{eq: reformulation of penal. q Gaussian}) with data $A_i$ and $X_B$ is that with the perturbed data $B_i$ on $5$ random data sets when $q>1$.} \label{t4}
	\begin{center}
		{\scriptsize
			\begin{tabular}{| c | c| c| c| c| c| c| c| c| c| c| c| c| c| c| c|} \hline
				\mc{1}{|c|}{} &\mc{5}{c|}{ $\gamma=0.1$ and $\epsilon=10^{-2}$}&\mc{5}{c|}{ $\gamma=0.1$ and $\epsilon=10^{-3}$}&\mc{5}{c|}{ $\gamma=0.1$ and $\epsilon=10^{-5}$}\\[3pt] \hline
				1.2 & 0.77 & 0.77 & 0.82 & 0.74 & 0.76 & 0.77 & 0.77 & 0.82 & 0.74 & 0.76 & 0.77 & 0.77 & 0.82 & 0.74 & 0.76 \\[3pt]
				1.1 & 1.54 & 1.53 &	1.61 & 1.50 & 1.53 & 1.54 & 1.53 &	1.61 & 1.50 & 1.53 & 1.54 & 1.53 &	1.61 & 1.50 & 1.53 \\[3pt]
				1.01 & 2.21 & 2.18 & 2.28 &	2.16 & 2.20 & 2.21 & 2.18 & 2.28 &	2.17 & 2.20 & 2.21 & 2.18 & 2.28 &	2.17 & 2.21 \\[3pt] \hline
				\mc{1}{|c|}{} &\mc{5}{c|}{ $\gamma=0.01$ and $\epsilon=10^{-2}$}&\mc{5}{c|}{ $\gamma=0.01$ and $\epsilon=10^{-3}$}&\mc{5}{c|}{ $\gamma=0.01$ and $\epsilon=10^{-5}$}\\[3pt] \hline
				1.2 & 2.13 & 2.11 &	2.22 & 2.08 & 2.12 & 2.13 & 2.12 &	2.22 & 2.08 & 2.12 & 2.13 & 2.12 &	2.22 & 2.08 & 2.12 \\[3pt]
				1.1 & 2.55 & 2.52 &	2.63 & 2.49 & 2.54 & 2.55 & 2.52 &	2.64 & 2.50 & 2.54 & 2.55 & 2.52 &	2.64 & 2.50 & 2.54 \\[3pt]
				1.01 & 2.68 & 2.64 & 2.76 &	2.62 & 2.67 & 2.68 & 2.65 & 2.77 &	2.63 & 2.67 & 2.68 & 2.65 & 2.77 &	2.63 & 2.67 \\[3pt] \hline
		\end{tabular}}
	\end{center}
\end{table}

From Tables \ref{t3}-\ref{t4}, we can observe that $\|X_B-X_A\|_F\le 4\epsilon$, where $X_A$ is the final estimated solution of the model (\ref{eq: reformulation of penal. q Gaussian}) with data $A_i$ and $X_B$ is that with the perturbed data $B_i$, for all the cases. The value $\|X_B-X_A\|_F/\epsilon$ tends to reduce if the regularization parameter $\gamma$ and $q$ are getting large.

To visualize the effect of $\gamma$, we create the following toy example:
$$
A_1=\left[\begin{array}{cc} 1 & 0 \\ 0 & 1 \end{array}\right],\;A_2=\left[\begin{array}{cc} 5 & 0 \\ 0 & 5 \end{array}\right],\;A_3=\left[\begin{array}{cc} 10 & 0 \\ 0 & 10 \end{array}\right].
$$
In this experiment, we set $q=0.5$ and $\epsilon=10^{-5}$.
\begin{eqnarray*}
	X_{A,1}=\left[\begin{array}{cc} 3.51060650 & 0 \\ 0 & 3.51060650 \end{array}\right] & X_{B,1}=\left[\begin{array}{cc} 3.51061752 & 0 \\ 0 & 3.51061752 \end{array}\right]\\
	X_{A,0.1}=\left[\begin{array}{cc} 4.43890037 & 0 \\ 0 & 4.43890037 \end{array}\right] & X_{B,0.1}=\left[\begin{array}{cc} 4.43891276 & 0 \\ 0 & 4.43891276 \end{array}\right]\\
	X_{A,0.01}=\left[\begin{array}{cc} 4.53771225 & 0 \\ 0 & 4.53771225 \end{array}\right] & X_{B,0.01}=\left[\begin{array}{cc} 4.53772477 & 0 \\ 0 & 4.53772477 \end{array}\right]\\
	X_{A,0}=\left[\begin{array}{cc} 4.54875843 & 0 \\ 0 & 4.54875843 \end{array}\right] & X_{B,0}=\left[\begin{array}{cc} 4.54877096 & 0 \\ 0 & 4.54877096 \end{array}\right],
\end{eqnarray*}
where $X_{A,\gamma}$ is the final estimated solution with the given $\gamma$ and data $A_i$ and $X_{B,\gamma}$ is the final estimated solution with the given $\gamma$ and the perturbed data $B_i$. The bigger the penalty parameter $\gamma$ is, the smaller the value of each diagonal entry of the solution is.

\section*{Acknowledgment} We thank the anonymous referees for providing useful comments and suggestions that have helped us to improve the presentation of the paper. The idea of the proof of Proposition \ref{prop: invariant} was suggested by a referee. 
\bibliographystyle{abbrv}

\end{document}